\documentclass[a4paper] {article}[12pt]

\usepackage[top=3cm, bottom=3cm, left=3.58cm, right= 3.58cm]{geometry}

\makeatletter
\renewcommand*\l@section{\@dottedtocline{1}{1.5em}{2.3em}}
\makeatother

\usepackage{amsfonts}
\usepackage{amssymb}
\usepackage[T1]{fontenc}

\usepackage{tikz}
\usetikzlibrary{calc}

\usepackage{CJK}
\usepackage{amsmath}
 
\usepackage{amsfonts}
\usepackage{amssymb}
\usepackage{amsthm}
\usepackage{amssymb}
\usepackage{enumerate}
\usepackage[calc]{picture}
\usepackage[all,cmtip]{xy}

\usepackage[mathscr]{eucal}
\usepackage{eqlist}

\usepackage{color}
\usepackage{abstract} 
\usepackage[T1]{fontenc}
 
\theoremstyle{plain}
\newtheorem{theorem}{Theorem}
\newtheorem{proposition}[theorem]{Proposition}
\newtheorem{lemma}[theorem]{Lemma}

\newtheorem{example}[theorem]{Example}
\newtheorem{corollary}[theorem]{Corollary}

\theoremstyle{definition}
\newtheorem{definition}{Definition}

\usepackage{etoolbox}
\newtheoremstyle{myrem}
 {3pt}
 {3pt}
 {\normalsize}
 { }
 {\itshape}
 {:}
 { }
 {}

 \theoremstyle{myrem}
 \newtheorem{remark}{Remark}
 \appto\remark{\leftskip\parindent}
 \appto\remark{\rightskip\parindent}

\usepackage{amsmath}

\numberwithin{equation}{section}
\numberwithin{theorem}{section}

\begin{document}

\begin{center}
{\Large{\textbf{Double  complexes  for  configuration  spaces      and     hypergraphs  on  
  manifolds
 }}}

 \vspace{0.58cm}
 
 Shiquan Ren 

\bigskip

\bigskip

 \parbox{24cc}{{\small

{\textbf{Abstract}.}  
In  this  paper,  we  consider  hypergraphs  whose  vertices  are  distinct  points  moving 
smoothly    
on  a  Riemannian  manifold  $M$.  
We  take  these  hypergraphs  as   graded  submanifolds of     configuration  spaces.  
We  construct  double  complexes  of  differential  forms  on  configuration  spaces.  
Then  we  construct      double  complexes   of     differential  forms  on  hypergraphs 
which  are  sub-double  complexes  of  the  double  complex  for  the  ambient  configuration  space.  
Among  these  double  complexes  for  hypergraphs,    the  infimum  double  complex  and  the  
supremum  double  complex  
are  quasi-isomorphic 
concerning   the  boundary  maps   induced  from   vertex  deletion  of  the  hyperedges. 
In  particular,  all  the  double  complexes     are  identical  if  the  hypergraph  
is  a  $\Delta$-submanifold  of  the  ambient  configuration  space.  
  }

\bigskip

}

\begin{quote}

\end{quote}

\end{center}

\vspace{1cc}

%

\section{Introduction}

\subsection{Configuration  spaces} 

Let  $M$  be  a  manifold.  
The  ordered  configuration  space ${\rm  Conf}_n(M)$  
is  the  manifold  consisting of  all  the  ordered  $n$-tuples  
$(x_1,\ldots,x_n)$  where  $x_1,\ldots,x_n$ are  distinct  points in  $M$ and  
the  unordered  configuration  space ${\rm  Conf}_n(M)/\Sigma_n$ 
is  the  orbit  space  of  the  $\Sigma_n$-action  on  ${\rm  Conf}_n(M)$.   
The  geometry  and  topology  of  configuration  spaces  has  been  extensively  studied  during  the last half century. 
For  example,
D. McDuff  \cite{mapping2}  in  1975, 
 W.  Fulton and R.  MacPherson  \cite{annals}  in  1994,  
P. Salvatore \cite{salvatore}  in 2004,  
T.  Church  \cite{inv}  in 2012,  
Y.  Baryshnikov, P.   Bubenik  and  M. Kahle  \cite{imrn1}  in 2014, 
T.  Church, J.  S. Ellenberg   and   B.  Farb  \cite{duke}  in 2015,  
A. Bianchi, J. Miller and J. C. H. Wilson  \cite{tams1}  in 2022,   etc.

 Throughout  the  study  of  the  geometry and  topology  of  configuration  spaces,  
 the   homology   and  cohomology      has  
 attracted  a  lot  of     attention. 
 In  recent  decades,  the   homology   and  cohomology    of  configuration  spaces 
 has been  extensively  investigated.  
  For  example,  C.-F. B\"odigheimer and F. R. Cohen  \cite{homol2} in 1987,  
    C.-F. B\"odigheimer,   F. Cohen   and L. Taylor 
 \cite{homol} in  1989, 
T.  Church  \cite{inv}  in 2012, T.  Church, J.  S. Ellenberg   and   B.  Farb  \cite{duke}  in 2015,  
  etc.

 Suppose  $M$  has  a  Riemannian  metric.  Besides  ${\rm  Conf}_n(M)$  
 and  ${\rm  Conf}_n(M)/\Sigma_n$,  for  each  $r\geq  0$  we  have  the  configuration spaces 
 of  hard  disks  and  hard spheres  of  radius  $r$  consisting  of  the $n$-tuples  of  
  distinct  points  in  $M$  with  pairwise  
 distance  greater than  $2r$.  
 The  configuration  spaces  of  hard  disks  and  hard spheres  
 are   motivated  by  physical considerations  and  
 have  significance  for  their   applications  in  phase  transitions,  motion  planning   and  optimal  packing.   Detailed  discussions  can  be  found  in  
  G.  Carlsson, J.  Gorham, M.  Kahle   and J.  Mason  \cite{phys}  in  2012,   
  Y.  Baryshnikov, P.   Bubenik  and  M. Kahle  \cite{imrn1}  in  2014,  
   H. Alpert, M.  Kahle  and  R. MacPherson \cite{confcomp}  in 2021, 
   H.   Alpert  and  Fedor Manin  \cite{gt}  in  2024,  etc.

\subsection{Chain  complexes for  hypergraphs}

Let  $V$  be  a  discrete  set  of  vertices.  
A  non-empty  finite  subset  $\sigma$ of  $V$  is  called  a  {\it  hyperedge}  on  $V$. 
If  $\sigma$  has  $n$  vertices,  then     $\sigma$  is  called   an  {\it $n$-hyperedge}.   
A    collection  $\mathcal{H}$  of  certain  hyperedges  on  $V$  is  called  a {\it hypergraph}.   If  each  hyperedge  in  $\mathcal{H}$  is  an  $n$-hyperedge,  
 then     $\mathcal{H}$   is called   an  {\it  $n$-uniform  hypergraph}   (cf.  \cite{berge}).  
 Moreover,  a  {\rm  directed  hyperedge} on  $V$  is  an  ordered  sequence  of  vertices  in  $V$  and    
 a  {\rm  hyperdigraph}  is  a  collection of  certain directed  hyperedges  on  $V$
 (cf.  \cite[Section~3.1]{hdg}).

An  (abstract)  simplicial  complex  $\mathcal{K}$   is  a  hypergraph   such  that  any  
 non-empty  subset  of  any  hyperedge  $\sigma$  in  $ \mathcal{K}$  is  still a  hyperedge  in $ \mathcal{K}$.    From  a  topological  point  of  view,  
 a  hypergraph  is  obtained  from  
 an  (abstract)   simplicial  complex  by  removing  certain  non-maximal  faces
 (cf.  \cite{parks}).

The  homology  and  cohomology      theory  of  hypergraphs  has  been  studied  
since  1990s.  In  1991,    A. D. Parks and S. L. Lipscomb  \cite{parks}  
considered  the  smallest  simplicial  complex  containing a  hypergraph  (called  the  
associated  simplicial  complex)  
and  used  the  homology  of  this simplicial  complex  to  investigate  the  hypergraph.  
In  1992,  F. R. K. Chung and R. L. Graham  \cite{chung}  
regarded  the  hyperedges  in   a  $k$-uniform  hypergraph  
as  $(k-1)$-chains  and  studied  the  mod  2  cohomology  by  considering  certain  induced  
group  structures.

 In  2019,  inspired  by  
the  path  complex  construction by  
A.  Grigor'yan,  Y.  Lin, Y.  Muranov  and   S.-T. Yau  \cite{lin2,lin6},  
S.  Bressan,  J.  Li,      J. Wu  and  the  present  author  
constructed  the  infimum  chain  complex  and  the 
supremum  chain  complex   for  a  hypergraph    in  \cite{h1}.     
 In  particular,  if  the  hypergraph  is  a  simplicial  complex,  
 then  both  the   infimum  chain  complex  and  the   supremum  chain  complex
 are  equal  to  the   usual  chain  complex.  
 In  2023,  D. Chen,    J.  Liu,  
  J. Wu   and G.-W.  Wei   constructed  
       chain  complexes,  homology  groups   and  Laplacians   for  hyperdigraphs  in  
  \cite{hdg},  which  has  significant  applications  in  data  science.

\subsection{Motivations}\label{subs-1.3-intro}

In  this  paper,  we  consider  hyper(di)graphs   on  manifolds,  i.e.  
  hyper(di)graphs  whose  vertices  are  lying  in a  manifold  $M$.  
  We  construct   double  complexes  for   hyper(di)graphs   on  manifolds.  
We  list  three  scenarios  that    hyper(di)graphs  on  manifolds 
as  well  as  the  double  complexes   are  prospectively  helpful:
 the  regular embedding of  manifolds,  the  motion  planning  with  constraints,  
  and  the  manifolds  learning  of  network data.

Firstly,  
 the canonical  bundles  over   hypergraphs  on  manifolds  and  
 their  characteristic  classes    in  the  cohomology    
might yield  new  topological  obstructions  
for  regular  embeddings  of   manifolds  in  Euclidean  spaces.

A  $k$-regular  embedding  of  a  manifold  into  a   Euclidean space
is  an  embedding  such  that  
 the  image  of   any distinct  $k$  points  are  linearly  independent.     
 The   regular  embedding   problem  is   initially  studied  
           by  K.  Borsuk \cite{Borsuk}   in  1957  and  has  connections  with    
            the   \v{C}eby\v{s}ev approximation  (cf. \cite{handel4} and \cite[pp. 237-242]{singer}, 
           the   Haar-Kolmogorov-Rubinstein  Theorem).  
            Some  topological  obstructions  for  regular  embeddings 
             using  the  Stiefel-Whitney  classes
            of  the  canonical  bundles  over  configuration  spaces    
            are   given               
           by  F.R. Cohen and D. Handel  \cite{cohen1}  in  1978 
           and  are   further  studied  by  P. Blagojevi$\check{\text{c}}$,   W. L{\"u}ck and G. Ziegler 
           \cite{high1}  in  2016   (cf.   Lemma~\ref{th-intro-1}~(1)).  
             Some  topological  obstructions  for  complex   regular  embeddings 
             using  the  Chern  classes
            of  the  complexification  of  the  canonical bundles  over  configuration  spaces    
            are   given               
           by P. Blagojevi$\check{\text{c}}$,  F. Cohen, W. L{\"u}ck and G. Ziegler
           \cite{high2}  in  2016 (cf.   Lemma~\ref{th-intro-1}~(2)).  
                      One  of  the  difficulties  of  the    regular  embedding  problem  
                      is  that 
                      the  cohomology  rings   of  configuration  spaces  
           as  well  as  the  characteristic  classes  of  the  canonical  bundles  over  configuration  spaces 
           are  difficult  to  calculate.

           Consider   any  $\Sigma_k$-invariant  
           $k$-uniform  hyperdigraph   $\vec{\mathcal{H}}_k(M)$  on  a  manifold $M$,  
           which  is  a  subspace  of   ${\rm  Conf}_k(M)$.  
           Let  $\mathcal{H}_k(M)= \vec{\mathcal{H}}_k(M) /\Sigma_k$.      
           Then  $\mathcal{H}_k(M)$  is  a  hypergraph  on  $M$,  
            which  is  a  subspace  of  ${\rm  Conf}_k(M)/\Sigma_k$.  
           The  canonical  bundle  over  ${\rm  Conf}_k(M)/\Sigma_k$   
           will induce  a  pull-back  bundle  over   $\mathcal{H}_k(M)$.    
           If there  exists  a  $k$-regular  embedding  of  $M$  in  a  Euclidean  space,  then  this  embedding 
            will  induce  a  classifying  map  of  the  canonical   bundle  over   $\mathcal{H}_k(M)$.
           Thus  the  characteristic  classes  of the  canonical  bundle  over   $\mathcal{H}_k(M)$ 
                      will  give  obstructions  for  regular  embeddings  of  $M$ 
           (cf.  Proposition~\ref{th-intro-2}).  
           If  the  cohomology  ring   of   $\mathcal{H}_k(M)$    
           as  well  as  the  characteristic  classes  of  the  canonical  bundle   over   $\mathcal{H}_k(M)$   
           can  be  calculated,   then  lower  bounds  of  the  dimension  of  the  ambient  
           Euclidean  space  for  regular  embeddings  of  $M$   could  be  investigated.

Secondly,   hypergraphs  on  manifolds  could  describe  motion  planning  problems 
 with   certain   constraints.

 Applications  of  configuration  spaces  in  motion  planning  problems  are  
 studied  by  M.  Farber \cite{farb08}  in  2008  and  \cite{Farber}  in  2017.   
 The  motion planning problem  requires   to  consider    continuous  motions  of  $k$  particles  from 
 a  start
state to a goal state that avoid   mutual collisions  as  well  as  collisions  with  obstacles.   
 The  motion planning question  by   M.  Farber   asks to find continuous sections of the  canonical  map 
 from the path space of  $X$  to  $X\times  X$  where  $X$  is  the  $k$-th  configuration  space
 of  $M$.  
In  addition,  if   we  require   that  the  $k$   particles 
   do  not  violate   certain  additional   constraints,  
 then  the  motion  planning  problem  requires  to  consider   continuous  paths 
 in  certain  $k$-uniform  hyper(di)graphs  on  manifolds
  that  are  subspaces  of  the  configuration  space.

For  example,   in  a  motion  planning  problem,   we   may   require   the  following  
additional  constraints:   
 three  points  among  the  $k$  particles    form  an   equilateral triangle;    
 three  points  among  the  $k$  particles  lie   on  a  common  line;    
 four  points  among  the  $k$  particles  lie   on  a  common circle;  
 the    position   of  the  $k$  particles   is   convex,  
 i.e.  all  the  $k$  points  lie  in  the  boundary  of  the  convex  hull;  etc. 
  Under   each  of   these  constraints,  
   the  motion planning  question  requires  
   to  consider  continuous  paths  in     
  a  hyper(di)graph  on  a  manifold,  which  is   
  contained  as  a  subspace  of  the  configuration  space.

Thirdly,     hypergraphs  on  manifolds   give  a  mathematical  model  for          
manifold  learning  of   network  data.

Hypergraphs  can  be  used  as   models  representing  the  data  from  higher-order  networks 
  (for  examples,  \cite{hdg, siam1, wulaoshi}). 
    In  various  scenarios,        the observed data lie on       a low-dimensional
manifold embedded in a higher-dimensional space  (for  examples,  \cite{mfdln2,mfdln1}). 
Manifold learning,   also  known  as non-linear dimension reduction, is a  fast-developing  area 
 whose  aim  is  to find the low dimensional structure of data in  order  to 
  reveal the geometric shape of high dimensional point clouds  (for  example,  refer  to  \cite{mfd-ln}). 
  To  study  point  clouds  with  relations  represented  by    networks 
  whose  points  lie  on       a low-dimensional
manifold embedded  in a high-dimensional space,     
it  is  reasonable  to   represent  the  point  clouds  by  hypergraphs  on  manifolds.

  With  the  help  of  the  model  of   hypergraphs  on  manifolds,  
   persistent  homology  and  Laplacians  can  be  applied  to  analyze  the  data  of  
    networks  with    vertices   on  manifolds.  
Persistent  homology and  Laplacians  are  important  tools    
 to  analyze  data  (for  examples,  \cite{  lap1,lap2, cph1, wulaoshi,hdg}).  
  From  a  mathematical  point  of view,  chain  complexes  as  well  as  double     complexes  
  give     foundations  for  homology  groups  and  Hodge-Laplacians. 
  To  analyze   the  higher-order  networks  whose    vertices  
  lie  on      a     
manifold  by  using  persistent  homology  and   Hodge-Laplacians,  
   the   chain  complexes  as  well  as  the   double     complexes
   give  a  significant    part  of  the   theoretical  foundations.  

\subsection{Main  results}

  In  this  paper,  we  take the   hypergraphs   with  vertices  
  lying  on     Riemannian  manifold    as  a  mathematical   model  for   the       scenarios
  in  Subsection~\ref{subs-1.3-intro}.   We  
  embed    the   hypergraphs   with  vertices  
 lying  on  a   Riemannian  manifold  as   submanifolds  of  the  configuration  spaces  of  
 the  Riemannian  manifold 
  and  construct  
  certain  double    complexes  of  differential  forms.

We  take  $V$  to  be  all  the  possible   embedded  subsets    of   a  Riemannian  manifold  $M$. 
We  consider  
the  hyperdigraphs  
$\vec{\mathcal{H}}(M)$   on  $V$  
as  well  as    the  hypergraphs   $\mathcal{H}(M)$  on  $V$.  
Let the  vertices  in  $V$  run  over  all  the  possible  positions  in $M$  smoothly.  
Then    $\vec{\mathcal{H}}(M)$   is  a  subspace  of  the  ordered  configuration  space 
$\bigcup_{n\geq  1}{\rm  Conf}_n(M)$  
and  $\mathcal{H}(M)$   is  a  subspace  of  the  unordered  configuration  space  
  $\bigcup_{n\geq  1}{\rm  Conf}_n(M)/\Sigma_n$.

 We  call  an  $\mathbb{N}$-indexed    family     of  manifolds  a  {\it  $\Delta$-manifold}  if  
 this   family  of   manifolds   has  a  $\Delta$-set  structure  such that   all  the    face  maps  
 $\partial_n^i$ 
  are  smooth  maps   for  any  $i=1,\ldots,n$  and  any  $n\in  \mathbb{N}$.    
 Note  that   $\bigcup_{n\geq  1}{\rm  Conf}_n(M) $  is  a   $\Delta$-manifold.  
 Moreover,  if  $M$  can  be  embedded  in  $\mathbb{R}$ hence  continuously  
  inherits a  total  order  from  
 $\mathbb{R}$,  then  
   $\bigcup_{n\geq  1}{\rm  Conf}_n(M)/\Sigma_n$   is  a   $\Delta$-manifold.    
   Let  $\Omega^\bullet(-)$  be  the  space  of  
 differential  forms with  exterior  derivative  $d$.   
 The  next  theorem   follows   directly   from  
  Theorem~\ref{pr-ab1}  and  Theorem~\ref{pr-ab122222}.  
 
  \begin{theorem}
 We  have  a     double  complex  
 \begin{eqnarray*}
 (\Omega^\bullet({\rm  Conf}_\bullet(M)),  d,\partial)
 \end{eqnarray*}
    with   
a  graded  group  action  of  $\Sigma_n$  on  $\Omega^\bullet({\rm  Conf}_n(M))$  for  each 
$n\geq  1$  which  is  commutative  with  $d$.  
 In  addition,  if  $M$  can  be  embedded  in  $\mathbb{R}$,  
 then  we  have   a   sub-double  complex   
  \begin{eqnarray*}
 (\Omega^\bullet({\rm  Conf}_\bullet(M)/\Sigma_\bullet),  d,\partial)
  \end{eqnarray*}
       consisting  of  the  $\Sigma_\bullet$-invariant  differential  forms.   
       \label{th-int-1}
 \end{theorem}

    Let  $\vec{\mathcal{H}}_n(M)$  be  the  collection  of  all the  directed  
    $n$-hyperedges  in  $\vec{\mathcal{H}}(M)$  and  
 let  $\mathcal{H}_n(M)$  be  the  collection  of  all the  $n$-hyperedges  in  $\mathcal{H}(M)$.  
 We  suppose  in  addition  that  
 each  $\vec{\mathcal{H}}_n(M)$  as  well  as   each  $\mathcal{H}_n(M)$  
 are  differentiable  manifolds  hence  
 $\vec{\mathcal{H}}(M)=\bigcup_{n\geq  1}\vec{\mathcal{H}}_n(M) $  is  a  graded  submanifold  of  
 $\bigcup_{n\geq  1}{\rm  Conf}_n(M)$   and  $\mathcal{H}(M)=\bigcup_{n\geq  1}  {\mathcal{H}}_n(M) $  is  a  graded  submanifold  of  
 $\bigcup_{n\geq  1}{\rm  Conf}_n(M)/\Sigma_n$.    
 This  assumption  is  natural  motivated  
 by  manifold  learning.   Under  this  assumption,  
 both  $\vec{\mathcal{H}}_n(M)$  and   $\mathcal{H}_n(M)$  
 can  be  studied  by  applying  the  tools  in  differential geometry and  the  topology
  of  manifolds.

   Let  $\Delta \vec{\mathcal{H}} (M)$  be  the  smallest  $\Delta$-manifold  containing  
 $\vec{\mathcal{H}} (M)$.  
 Let $\delta\vec{\mathcal{H}} (M)$  be  the  largest  $\Delta$-manifold  contained  in  
 $\vec{\mathcal{H}} (M)$.  
 Let  $\Delta  {\mathcal{H}} (M)$  be  the  smallest  $\Delta$-manifold  containing  
 $ {\mathcal{H}} (M)$.  
 Let $\delta {\mathcal{H}} (M)$  be  the  largest  $\Delta$-manifold  contained  in  
 $ {\mathcal{H}} (M)$.   The  next theorem  is  the  ordered  version  of  the     main  result.

 \begin{theorem}[Main  result  I]
For  any   hyperdigraph  $\vec{\mathcal{H}}(M)$  on  $M$,  
 we have  surjective  homomorphisms  of  double  complexes 
 \begin{eqnarray*} 
&&(\Omega^\bullet( \Delta \vec{\mathcal{H}}(M)),  d,\partial)\longrightarrow   ({\rm  Sup}^\bullet(\vec{\mathcal{H}} (M)), d,\partial)
\nonumber\\
&&\overset{q}{\longrightarrow} ( {\rm  Inf}^\bullet(\vec{\mathcal{H}} (M)),   d,\partial)
\longrightarrow  (\Omega^\bullet( \delta \vec{\mathcal{H}}(M)),  d,\partial).  
 \end{eqnarray*}
 such  that 
 \begin{enumerate}[(1)]
 \item 
 $q$  is  a  quasi-isomorphism  with  respect to $\partial$;
 \item 
   $  \vec{\mathcal{H}}(M) $  is  a  $\Delta$-submanifold  of  ${\rm   Conf}_\bullet(M)$
  iff  all   the three  homomorphisms  are  the  identity.   
  \end{enumerate}
    \label{th-int-2}
\end{theorem}

 The  next theorem  is  the  unordered  version  of  the     main  result.  

\begin{theorem}[Main  result  II]
If  $M$  can  be  embedded  in  $\mathbb{R}$  
  and  consequently  has  a  continuous  total  order,  then  for  any   hypergraph  ${\mathcal{H}}(M)$  on  $M$,  
 we have  surjective  homomorphisms  of  double  complexes 
 \begin{eqnarray*} 
&&(\Omega^\bullet( \Delta  {\mathcal{H}}(M)),  d,\partial)\longrightarrow   ({\rm  Sup}^\bullet( {\mathcal{H}} (M)), d,\partial)
\nonumber\\
&&\overset{q}{\longrightarrow} ( {\rm  Inf}^\bullet( {\mathcal{H}} (M)),   d,\partial)
\longrightarrow  (\Omega^\bullet( \delta {\mathcal{H}}(M)),  d,\partial).  
 \end{eqnarray*}
 such  that  
 \begin{enumerate}[(1)]
 \item
 $q$  is  a  quasi-isomorphism  with  respect to $\partial$;
 \item 
   $  {\mathcal{H}}(M) $  is  a  $\Delta$-submanifold  of  ${\rm   Conf}_\bullet(M)/\Sigma_\bullet$
  iff  all   the three  homomorphisms  are  the  identity.   
  \end{enumerate}
    \label{th-int-3}
\end{theorem}

  Theorem~\ref{th-int-2}  follows  directly  from  Theorem~\ref{th-main1}  and  Theorem~\ref{th-main1aa}.  
  Theorem~\ref{th-int-3}  follows  directly  from  Theorem~\ref{th-main2}  and  Theorem~\ref{th-main1bb}.

\subsection{Organization} 

 Section~\ref{sec-saa}  and  Section~\ref{sec3}  
 are  auxiliaries  for  later  sections.  
In  Section~\ref{sec-saa},  we  review the   infimum  and  the  supremum  chain  complexes. 
In  Section~\ref{sec3},  
we  give  some  double  complexes   for   $\Delta$-manifolds  and   their 
 graded  submanifolds.   
 We  prove  a  quasi-isomorphism concerning  the  boundary  map  induced  by  the  
 face  maps  of  the  $\Delta$-set  structure  from    
 the  supremum  double  complex to  the  infimum  double  complex.

 In  Section~\ref{sec-conf-hpg},  
 we  introduce   the  notions  of  configuration  spaces  and hypergraphs.  
 We  give  some  examples  for  hypergraphs  on  discrete sets,   circles  and  spheres.   
  In  Section~\ref{sec5},  we  construct  the  automorphism  groups  of  
  hypergraphs   and  prove  that  the  group  action  is   faithful.  
  We  calculate  some  examples  of  the  automorphism groups  
  for  hypergraphs  on  discrete sets,    circles  and  spheres.  
  In  Section~\ref{sec-10}, 
  we  prove   that  the  automorphism  group  of  a  hypergraph  
  is  a  subgroup  in  the  intersection  of  the  automorphism  
  groups  of  the  (lower-)associated  simplicial complexes  and  the  (lower-)associated  
  independence  hypergraphs. 
  In  Section~\ref{bundle},  we consider   the  associated vector  bundles  for  covering  maps  from   
   hyperdigraphs  to  the  underlying  hypergraphs.    
 We  estimate  the  order  of  these      vector  bundles.  
 In  Section~\ref{sec-regular},  we   use  hypergraphs  on  manifolds  
 to  give   obstructions  for  $k$-regular  embeddings.

 In   Section~\ref{sec6},  we construct  the  double  complexes  for  configuration  spaces 
 and  prove   Theorem~\ref{th-int-1}   with  the  help  of    Subsection~\ref{ssec-3.1}.  
In  Section~\ref{sec7},   we  construct  double  complexes  for  hypergraphs  on  $M$    
and  prove  Theorem~\ref{th-int-2}   and  Theorem~\ref{th-int-3}    with  the  help  of 
    Subsections~\ref{ss2.3a}
  and  \ref{ssec-3.3}.

\section{The  infimum  and  supremum  chain  complexes}\label{sec-saa}

In  this  section,  we  review the  infimum  as  well  as  the   supremum  chain  complexes   in  \cite{h1}  and  
  give    some   preliminaries   for  Subsection~\ref{ss2.3a}.

 Let $\mathbb{N}=0,1,2,\ldots$   be   the  natural  numbers.  
 Let  $C = (C_n,\delta_n)_{n\in  \mathbb{N}}$  be  a  chain  complex,  where  $C_n$  is  a  vector  space  and  $\delta_n:  C_n\longrightarrow  C_{n-1}$  is  a  linear  map  satisfying 
  $\delta_{n}\delta_{n+1}=0$  
for  each  $n\in  \mathbb{N}$.  
Let  $D =(D_n )_{n\in  \mathbb{N}}$  be  a  graded  subspace  of  $C$.  
The  {\rm  infimum  chain  complex}  is  a  sub-chain  complex of  $C$  given  by  
(cf.  \cite[Section~2,  Definition~1  and  (2.2)]{h1})
\begin{eqnarray}\label{eq-2.z.z.1}
{\rm  Inf}_n (D,C)= D_n\cap  \delta_n^{-1}  D_{n-1}
\end{eqnarray}
and  the  {\it  supremum  chain  complex}  is  a  sub-chain  complex of  $C$  given  by  
(cf.  \cite[Section~2,  Definition~2  and  (2.3)]{h1})
 \begin{eqnarray}\label{eq-2.z.z.2}
{\rm  Sup}_n (D,C)= D_n+  \delta_{n+1}   D_{n+1}.  
\end{eqnarray}
It  is  proved  in \cite[Proposition~2.4]{h1}  and  \cite[Proposition~2.3]{wulaoshi}  that  the  canonical  inclusions  
\begin{eqnarray}\label{eq-0.mpa}
\theta: {\rm  Inf}_\bullet (D,C)\longrightarrow  {\rm  Sup}_\bullet (D,C) 
\end{eqnarray}    
 is  a  quasi-isomorphism  of  chain  complexes.

   The  chain  complexes  and    chain  maps  in   
 (\ref{eq-0.mpa})  only  depend  on  $D$  and  
 do  not  depend  on  the  choice  of  the  ambient  chain  complex  $C$.   Precisely,    
 if  $C$  and  $C'$  are  two  chain  complexes  containing  $D$  as  a  graded  subspace  
 such  that  the   boundary  maps  
 $\partial $  of  $C$   and  $\partial'$  of   $C'$     coincide  on  $C\cap C'$,  
 then  ${\rm  Inf}_\bullet (D,C)={\rm  Inf}_\bullet (D,C')$  and    ${\rm  Sup}_\bullet (D,C)={\rm  Sup}_\bullet (D,C')$.  
 Thus  we  simply  write   ${\rm  Inf}_\bullet (D,C)$  as  ${\rm  Inf}_\bullet (D)$  
 and  write  ${\rm  Sup}_\bullet (D,C) $  as  ${\rm  Sup}_\bullet (D) $.  

\begin{lemma}\label{le-0.zam1}
The  chain  map  of  the  quotient  chain  complexes  
\begin{eqnarray}\label{eq-1.mpa}
q:  C/ {\rm  Sup}_\bullet (D)  \longrightarrow  C/ {\rm  Inf}_\bullet (D)
\end{eqnarray}
induced  by  
(\ref{eq-0.mpa})  is  a  surjective  quasi-isomorphism.       
\end{lemma}

\begin{proof}
Since  (\ref{eq-0.mpa})  is  injective,  its  induced  chain  map  (\ref{eq-1.mpa})  
is  surjective.  
The  induced  homomorphism  of  homology groups  
\begin{eqnarray*} 
q_*:  H_n( C/ {\rm  Sup}_\bullet (D))  \longrightarrow   H_n(C/ {\rm  Inf}_\bullet (D))
\end{eqnarray*}
is  an  isomorphism  
\begin{eqnarray*}
 H_n( C/ {\rm  Sup}_\bullet (D))  &=&  
 \frac{{\rm  Ker}(\delta_n)/ {\rm  Ker}(\delta_n\mid _{D_n+\delta_{n+1}D_{n+1}})}
 {{\rm  Im} (\delta_{n+1})/  {{\rm  Im} (\delta_{n+1}\mid_{D_{n+1} +\delta_{n+2}  D_{n+2}})}}\\
 &=&\frac{{\rm  Ker}(\delta_n)/   ( \delta_{n+1}D_{n+1}+{\rm  Ker}(\delta_n) \cap  D_n)}
 {{\rm  Im} (\delta_{n+1})/  \delta_{n+1}( D_{n+1} )}\\
 &\cong&\frac{{\rm  Ker}(\delta_n)/   ( {\rm  Ker}(\delta_n) \cap  D_n)}
 {{\rm  Im} (\delta_{n+1})/ ( {\rm  Ker}(\delta_n) \cap  D_n \cap \delta_{n+1}( D_{n+1} ))}\\
 &=&\frac{{\rm  Ker}(\delta_n)/     {\rm  Ker}(\delta_n\mid_{D_n\cap  \delta_n^{-1}  D_{n-1}}) }
 {{\rm  Im} (\delta_{n+1})/   {\rm  Im}(\delta_{n+1}\mid_{D_{n+1}\cap  \delta_{n+1}^{-1}  D_{n }})  }\\
 &=&  H_n( C/ {\rm  Inf}_\bullet (D))   
\end{eqnarray*}
for  each  $n\in \mathbb{N}$.  
Thus  (\ref{eq-1.mpa})    is  a  quasi-isomorphism.  
\end{proof}


\section{$\Delta$-manifolds  and  double  complexes}\label{sec3}

In  this  section,   we  introduce  the  definitions  of  
$\Delta$-manifolds,  graded  submanifolds  of  $\Delta$-manifolds,  
and  associated  $\Delta$-manifolds  for   graded  submanifolds  of  $\Delta$-manifolds. 
 We  construct  a  double  complex  for  a  $\Delta$-manifold,  
an  infimum  double  complex  as  well  as  a  supremum  double  complex  for  a  graded  
submanifold   of  a  $\Delta$-manifold,  and    a  double  complex 
for  the   (lower-)associated  $\Delta$-manifold.  
  This  section  gives  auxiliary  lemmas  for  the  proofs  in  Section~\ref{sec6}
   and  Section~\ref{sec7}.  

\subsection{Double  complex  for  $\Delta$-manifolds}\label{ssec-3.1}
\begin{definition}\label{def-7.1}
\begin{enumerate}[(1)]
\item
A  $\Delta$-set    is  a  sequence  of  sets  $S_\bullet=(S_n)_{n\in\mathbb{N}}$  with  face  maps 
$\partial_i^n:  S_n\longrightarrow  S_{n-1}$,  $i=0,1,\ldots,n$,  which  satisfy the $\Delta$-identity
\begin{eqnarray*}
\partial^{n-1}_i \partial^n_j= \partial ^{n-1}_{j-1} \partial  ^n_i 
\end{eqnarray*}
  for  any   $i,j\in\mathbb{N}$  such  that     $0\leq  i<j\leq  n$;    
 \item
A  {\it  $\Delta$-manifold}     is  a $\Delta$-set  $A_\bullet= (A_n)_{n\in\mathbb{N}}$
 where  for  every  $n\in  \mathbb{N}$,  $A_n$  is  a  smooth  manifold  and  all the  face  maps 
 $\partial_n^i:  A_n\longrightarrow  A_{n-1}$   are  smooth  maps;
 \item
 A  {\it  $\Delta$-submanifold}     of  a   {\it  $\Delta$-manifold}   $A_\bullet$   is  
 a  graded  subset   $A'_\bullet= (A'_n)_{n\in\mathbb{N}}$  of  $A_\bullet$  such  that  
 (i)  $\partial^n_i(A'_n)\subseteq  A'_n$ for  any    $n\in \mathbb{N}$  and  any   $0\leq  i \leq  n$, 
 and  
  (ii) $A'_n$  is  a  smooth  submanifold   of  $A_n$ for  any  $n\in  \mathbb{N}$.  
 \end{enumerate}
\end{definition} 

\begin{remark}
The  definition  of  $\Delta$-manifolds  in  Definition~\ref{def-7.1}  does  not  require  the  existence  of  
degeneracies  hence  it  is  weaker  than  the  definition  of  simplicial  manifolds  (cf.   \cite[Section~1]{simmfd}). 
\end{remark}

Let  $\Omega^k(A_n)$  be  the  space  of  differential  $k$-forms  on  $A_n$.  
Let  $d:  \Omega^k(A_n)\longrightarrow \Omega^{k+1}(A_n)$  
be  the  usual  exterior  derivative  of  differential   forms on   $A_n$.   Let  $\partial:   \Omega^k(A_{n-1})\longrightarrow  
 \Omega^k(A_{n})$,   where   $\partial= \sum_{i=0}^n (-1)^i (\partial_i^n)^*$,   be  the  alternative  sum 
of  the  induced  pull-back  maps  $(\partial_i^n)^*:  \Omega^{\bullet}(A_{n -1})\longrightarrow   \Omega^{\bullet}(A_{n})$   of  the face  maps  $\partial_i^n:  A_n\longrightarrow  A_{n-1}$  of  $A_\bullet $.  

\begin{lemma}\label{le-2.1}
For  any  $\Delta$-manifold  $A_\bullet$,  we  have  
$d^2=0$,  $\partial^2=0$  and  $\partial d =   d \partial$.  
Thus  $(\Omega^\bullet(A_\bullet),  d,\partial)$  is  a  double  complex. 
\end{lemma}

   \begin{proof}
Since  $d$  is  the   usual  exterior  derivative  of  differentiable  forms,  it  satisfies $d^2=0$.  
By  the  $\Delta$-identity  of  the  face  maps, 
\begin{eqnarray*}
\partial^2&=& \Big(\sum_{i=0}^{n+1}   (-1)^i (\partial _i^{n+1})^* \Big) \circ\Big(\sum_{j=0}^n   (-1)^j (\partial _j^n)^* \Big)\\
&=&\sum_{i=0}^{n+1}\sum_{j=0}^{i-1}  (-1)^{i+j}  ( \partial _j^n\partial _i^{n+1})^* +\sum_{i=0}^{p}\sum_{j=i}^n  (-1)^{i+j} ( \partial _j^n\partial _i^{n+1})^* \\
&=&  \sum_{i=0}^{n+1}\sum_{j=0}^{i-1}  (-1)^{i+j}  ( \partial _{i-1}^n\partial _j^{n+1})^*+\sum_{i=0}^{n}\sum_{j=i}^n  (-1)^{i+j} ( \partial _j^n\partial _i^{n+1})^* \\
&=&  \sum_{j=0}^{n}\sum_{i-1=j}^{n}  (-1)^{(i-1)+j+1}  ( \partial _{i-1}^n\partial _j^{n+1})^* +
\sum_{i=0}^{n}\sum_{j=i}^n  (-1)^{i+j} ( \partial _j^n\partial _i^{n+1})^*\\
&=&0.  
\end{eqnarray*}
Since  $(\partial _i^n)^*$  is  the  pull-back  of  differential   forms,  we  have  $(\partial _i^n)^* d=d (\partial _i^n)^*$,  which  implies  
\begin{eqnarray*}
\partial  d 
=  \sum_{i=0}^n (-1)^i (\partial _i^n)^*   d  
=  \sum_{i=0}^n (-1)^i  d(\partial _i^n)^*    
=  d \partial.
\end{eqnarray*} 
Consequently,  $(\Omega^\bullet(A_\bullet),  d,\partial)$  is  a  double  complex.   
\end{proof}

\begin{lemma}\label{le-6-987}
For  any  $\Delta$-manifold  $A_\bullet$  and  any    $\Delta$-submanifold  $A'_\bullet$  of  $A_\bullet$,  
the  canonical  inclusion  $\iota:  A'_\bullet\longrightarrow  A_\bullet$  induces  
a  surjective  chain  map     
\begin{eqnarray}\label{eq-7.mza}
\iota^\#:  (\Omega^\bullet(A_\bullet),  d,\partial)\longrightarrow  (\Omega^\bullet(A'_\bullet),  d,\partial)
\end{eqnarray}
of  double  complexes.  Consequently,  
\begin{eqnarray}\label{eq-7.zvab}
(\Omega^\bullet(A'_\bullet),  d,\partial)\cong   (\Omega^\bullet(A_\bullet),  d,\partial) /  {\rm  Ker}(\iota^\#)
\end{eqnarray}
is  a  quotient  double  complex of  $(\Omega^\bullet(A_\bullet),  d,\partial)$  modulo  the  kernel  of (\ref{eq-7.mza}).  
\end{lemma}

\begin{proof}
The  canonical  inclusion    $\iota:  A'_\bullet\longrightarrow  A_\bullet$
induces    pull-back  maps  of  differential  forms  
\begin{eqnarray}\label{eq-7.0va}
\iota^\#:  \Omega^\bullet(A_\bullet)\longrightarrow \Omega^\bullet(A'_\bullet)
\end{eqnarray}
by    restricting   the  differential  forms  on $A_\bullet$  to  $A'_\bullet$.    
Hence  (\ref{eq-7.0va})  is  surjective  and  $d \iota^\#= \iota^\#  d$.   
 By  Definition~\ref{def-7.1}~(3),  we  have  
  $(\partial^n_i )^*\iota^\#= \iota^\#   (\partial^n_i)^*$   for  any    $n\in \mathbb{N}$  and  any   $0\leq  i \leq  n$.  
  Hence     $\partial \iota^\#= \iota^\#  \partial$.     
  Therefore,  (\ref{eq-7.0va})  gives  a    surjective  chain  map  (\ref{eq-7.mza}) of  double  
  complexes.    
\end{proof}

\subsection{Double  complexes  for  graded  submanifolds  of  $\Delta$-manifolds}
\label{ss2.3a}

\begin{definition}\label{def-2.3mko}
Let  $A_\bullet=(A_n)_{n\in \mathbb{N}} $  be  a  $\Delta$-manifold.  
A  {\it  graded  submanifold}   $B_\bullet=(B_n)_{n\in \mathbb{N}} $   of  $A_\bullet$  is  
a  sequence  of   manifolds  such  that  $B_n$  is  a  submanifold   of  $A_n$  with  smooth   embeddings   $\epsilon_n:  B_n\longrightarrow  A_n$  for  each  $n\in \mathbb{N}$.  
\end{definition}

Let   $B_\bullet$  be  a  graded  submanifold  of  a  $\Delta$-manifold  $A_\bullet$ 
with  smooth  embeddings  $\epsilon_\bullet$.  This  induces  a   pull-back  of  
differential  forms  
\begin{eqnarray*}
 \epsilon_n ^\#:  \Omega^\bullet (A_n)\longrightarrow   \Omega^\bullet (B_n)
\end{eqnarray*}
    restricting  a  differential  form  on  $A_n$  to  $B_n$.  
We  have a  graded  subspace  
$
 ({\rm  Ker}  \epsilon_n ^\#) _{n\in  \mathbb{N}}
$,  
which  consists  of  all  the  differential  forms  on  $A_\bullet$  that  vanish  when
  restricted  to  $B_\bullet$,  
of  the   chain  complex  
$(\Omega^\bullet(A_n),\partial)_{n\in\mathbb{N}}$  
with   boundary  map  $\partial:  \Omega^\bullet(A_{n-1})\longrightarrow  \Omega^\bullet(A_n)$.  
By  (\ref{eq-2.z.z.1})  and  (\ref{eq-2.z.z.2})  respectively,  we  define    
the  infimum  double  complex  of  $\epsilon_\bullet$   as 
\begin{eqnarray*}
{\rm  Inf}_n(\Omega^\bullet(\epsilon_\bullet)) &=& {\rm  Inf}_n( {\rm  Ker}  \epsilon_\bullet ^\#,  \Omega^\bullet(A_\bullet))  \nonumber \\
&=&  {\rm  Ker}  \epsilon_n ^\#   \cap  \partial^{-1} {\rm  Ker}  \epsilon_{n+1}  ^\#  
\label{eq-lzqp1}
\end{eqnarray*}
and  define  the  supremum  double  complex of  $\epsilon_\bullet$   as  
\begin{eqnarray*}
{\rm  Sup}_n(\Omega^\bullet(\epsilon_\bullet)) &=& {\rm  Sup}_n( {\rm  Ker} \epsilon_\bullet^\#,  \Omega^\bullet(A_\bullet))  \nonumber\\
&=&  {\rm  Ker}  \epsilon_n ^\#  +  \partial  {\rm  Ker} \epsilon_{n-1}^\#.   
\label{eq-lzqp2}
\end{eqnarray*}
Then      
\begin{eqnarray*}
({\rm  Inf}_n(\Omega^\bullet(\epsilon_\bullet)),\partial)_{n\in \mathbb{N}}, ~~~~~~
({\rm  Sup}_n(\Omega^\bullet(\epsilon_\bullet)),\partial)_{n\in \mathbb{N}} 
\end{eqnarray*}  
are  sub-chain  complexes  of  $(\Omega^\bullet(A_n),\partial)_{n\in\mathbb{N}}$  such  that  
\begin{enumerate}[(1)]
\item
as  vector  spaces,     we  have  inclusions 
\begin{eqnarray*}
{\rm  Inf}_n(\Omega^\bullet(\epsilon_\bullet)) \subseteq     {\rm  Ker}  \epsilon_n ^\# \subseteq {\rm  Sup}_n (\Omega^\bullet(\epsilon_\bullet))\subseteq  \Omega^\bullet(A_n); 
\end{eqnarray*}
 \item
   by  (\ref{eq-0.mpa}),     
the  canonical  inclusion  
\begin{eqnarray}\label{eq-kk.1}
\theta:  {\rm  Inf}_n(\Omega^\bullet(\epsilon_\bullet)) \longrightarrow  {\rm  Sup}_n (\Omega^\bullet(\epsilon_\bullet))
\end{eqnarray}
is  a  quasi-isomorphism  of  chain  complexes.  
\end{enumerate}
Define  the  infimum  double  complex  of  $B_\bullet$  as 
\begin{eqnarray}\label{eq-mm1n}
{\rm  Inf}^n(\Omega^\bullet(B_\bullet))= \Omega^\bullet(A_n) /  {\rm  Sup}_n(\Omega^\bullet(\epsilon_\bullet))
\end{eqnarray}
and  define   the  supremum  double  complex  of  $B_\bullet$  as 
\begin{eqnarray}\label{eq-mm2n}
{\rm  Sup}^n(\Omega^\bullet(B_\bullet))= \Omega^\bullet(A_n)/  {\rm  Inf}_n(\Omega^\bullet(\epsilon_\bullet)).  
\end{eqnarray}
Then     
\begin{eqnarray*}
({\rm  Inf}^n(\Omega^\bullet(B_\bullet)),  \partial)_{_{n\in\mathbb{N}}},
~~~~~~~
({\rm  Sup}^n(\Omega^\bullet(B_\bullet)),  \partial)_{_{n\in\mathbb{N}}}
\end{eqnarray*}
 are  quotient  chain  complexes  of   $(\Omega^\bullet(A_n),\partial)_{n\in\mathbb{N}}$ 
 such  that  
 \begin{enumerate}[(1)]
\item
     we  have  quotient  linear maps  of  vector  spaces   
\begin{eqnarray}\label{eq-mkaoq1}
\Omega^\bullet(A_n)\longrightarrow  {\rm  Sup}^n(\Omega^\bullet(B_\bullet))
\longrightarrow  \Omega^n(B_\bullet)\longrightarrow 
{\rm  Inf}^n(\Omega^\bullet(B_\bullet)); 
\end{eqnarray}
 \item
   by  (\ref{eq-1.mpa}),      
the  canonical  quotient  linear  map    
\begin{eqnarray}\label{eq-mzvoaq}
q:   {\rm  Sup}^n(\Omega^\bullet(B_\bullet))
\longrightarrow   
{\rm  Inf}^n (\Omega^\bullet(B_\bullet))
\end{eqnarray}
is  a  quasi-isomorphism  of  chain  complexes  with  respect to $\partial$.  
\end{enumerate}

\begin{lemma}\label{le-hqlca7}
The  spaces  ${\rm  Inf}^n(\Omega^\bullet(B_\bullet))$  and  ${\rm  Sup}^n(\Omega^\bullet(B_\bullet))$  
as  well  as  the  linear  map  $q$  in  (\ref{eq-mzvoaq})  do  not  depend  on  the  choice  of  the  ambient  $\Delta$-manifold  $A_\bullet$.  
\end{lemma}

\begin{proof}
Let  $A_\bullet$  and  $A'_\bullet$  be  $\Delta$-manifolds  such that  their  face  maps  
coincide   on  the  intersection  $A_\bullet  \cap    A' _\bullet$. 
Then   $A_\bullet  \cap    A' _\bullet$   is   also  a  $\Delta$-manifold.  
Let  $B_\bullet$  be  a  graded  submanifold  of  $A_\bullet$  as  well as  $A'_\bullet$.  
Then  $B_\bullet$  is  also  a  graded  submanifold   of   $A_\bullet  \cap    A' _\bullet$.  
For  each  $n\in  \mathbb{N}$,  we  have  surjective  linear  maps  
\begin{eqnarray*}
\xymatrix{
\Omega^\bullet(A_n)\ar[rd] &&&&\\
&\Omega^\bullet(A_n\cap A'_n)\ar[r]& {\rm  Sup}_n(\Omega^\bullet(B_\bullet))\ar[r]  
&  \Omega^\bullet(B_n)  \ar[r]  &  {\rm  Inf}_n(\Omega^\bullet(B_\bullet)). \\
\Omega^\bullet(A'_n)\ar[ru] &&&&
}
\end{eqnarray*}
Consequently,   the  later  three  vector  spaces  and   the   later two  linear  maps  in  (\ref{eq-mkaoq1})   do  not  depend  on  the  choices  of  $A_\bullet$  or  $A'_\bullet$.  
\end{proof}

\begin{remark}
Nevertheless,  the  spaces  $ {\rm  Inf}_n(\Omega^\bullet(\epsilon_\bullet)) $ 
and  $  {\rm  Sup}_n (\Omega^\bullet(\epsilon_\bullet))$
as  well  as   the  map  $\theta$  in  (\ref{eq-kk.1})  
   depend  on  the  choice  of  $A_\bullet$.   The  independence  from   $A_\bullet$ 
comes  after   taking  the  quotient  spaces  in  (\ref{eq-mm1n})  and  (\ref{eq-mm2n}).     
\end{remark}

\begin{lemma}\label{le-1.1}
Let  $A_\bullet$  be  a  $\Delta$-manifold  and    
let  $B_\bullet$  be  a  graded  submanifold  of  $A_\bullet$ 
 with  embeddings  $\epsilon_\bullet$.   
Then   
  we  have 
  \begin{enumerate}[(1)]
  \item
   sub-double  complexes  
  \begin{eqnarray}\label{eq-lzpaqgj}
 ( {\rm  Inf}_\bullet(\Omega^\bullet(\epsilon_\bullet)),  d,   \partial ), ~~~~~~   ( {\rm  Sup}_\bullet(\Omega^\bullet(\epsilon_\bullet)),  d,   \partial )
  \end{eqnarray}
  of   $(\Omega^\bullet(A_\bullet),  d,\partial)$    such  that  the  canonical  inclusion  
 (\ref{eq-kk.1})  is  a  quasi-isomorphism  with respect  to  $\partial$;    
\item
 quotient  double  complexes  
   \begin{eqnarray}\label{eq-oabnbg}
 ( {\rm  Inf}^\bullet(\Omega^\bullet(B_\bullet) ),  d,   \partial ), ~~~~~~   ( {\rm  Sup}^\bullet(\Omega^\bullet(B_\bullet) ),  d,   \partial )
  \end{eqnarray}
of   $(\Omega^\bullet(A_\bullet),  d,\partial)$     such  that  the  canonical  quotient map  
(\ref{eq-mzvoaq})  is  a  quasi-isomorphism  with respect  to $\partial$.  
\end{enumerate}
\end{lemma}
\begin{proof}
The  exterior  derivative   $d$  and   $ \epsilon_\bullet ^\#$  commutes  
\begin{eqnarray*}
d  \epsilon_n ^\#=   \epsilon_n ^\# d
\end{eqnarray*} 
for  any  $n\in  \mathbb{N}$.  
Thus  the  restrictions  of  the  exterior  derivative  of  $\Omega^\bullet(A_n)$  induce   
  well-defined  exterior  derivatives  
\begin{eqnarray}\label{eq-mnmlao1}
d:  {\rm  Inf}_\bullet(\Omega^k(\epsilon_\bullet))\longrightarrow  {\rm  Inf}_\bullet(\Omega^{k+1}(\epsilon_\bullet)) 
\end{eqnarray}
and  
\begin{eqnarray}\label{eq-mnmlao2}
d:  {\rm  Sup}_\bullet(\Omega^k(\epsilon_\bullet))\longrightarrow  {\rm  Sup}_\bullet(\Omega^{k+1}(\epsilon_\bullet))    
\end{eqnarray}
for  any  $k\in \mathbb{N}$.  
Hence  (\ref{eq-lzpaqgj})  are  sub-double     complexes  of   $(\Omega^\bullet(A_\bullet),  d,\partial)$.    
 The  canonical  inclusion   $\theta$  in  
 (\ref{eq-kk.1})  is  a  quasi-isomorphism  of  chain  complexes  with respect  to  $\partial$.  
From   (\ref{eq-mnmlao1})  we  have  an  induced     map   of  the  quotient double  complex
\begin{eqnarray*}
d:  {\rm  Sup}^\bullet(\Omega^k(B_\bullet) )\longrightarrow  
{\rm  Sup}^\bullet(\Omega^{k+1}(B_\bullet) )
\end{eqnarray*}
 and  from   (\ref{eq-mnmlao2})  we  have  an  induced     map    of  the  quotient double  complex
\begin{eqnarray*}
d:  {\rm  Inf}^\bullet(\Omega^k(B_\bullet) )\longrightarrow  
{\rm  Inf}^\bullet(\Omega^{k+1}(B_\bullet) )
\end{eqnarray*}
for  any  $k\in \mathbb{N}$.   
Hence  (\ref{eq-oabnbg})  are  quotient  double  complexes  of    $(\Omega^\bullet(A_\bullet),  d,\partial)$.  The  canonical  quotient  map   $q$  in   (\ref{eq-mzvoaq})   is  a  quasi-isomorphism  of  chain  complexes  with respect  to  $\partial$.
\end{proof}

\begin{corollary}
\label{le-2.99a}
In   Lemma~\ref{le-1.1},  
If  the  graded  submanifold  $B_\bullet$   of  $A_\bullet$  is  a  $\Delta$-submanifold,  then  
both  the   double  complexes  in  (\ref{eq-lzpaqgj})  are  equal to 
\begin{eqnarray*}
( {\rm  Ker}  \epsilon_\bullet ^\#,   d,  \partial)
\end{eqnarray*}
and  both  the  double  complexes  in  (\ref{eq-oabnbg})  are  equal  to  
\begin{eqnarray*}
( \Omega^\bullet(B_\bullet),    d,  \partial).  
\end{eqnarray*}
In this  case,  both the  quasi-isomorphisms  $\theta$  in  (\ref{eq-kk.1})  and  $q$  in   (\ref{eq-mzvoaq})  are  the   identity  maps.  
\end{corollary}

\begin{proof}
The  corollary  follows  from  Lemma~\ref{le-6-987}  and  Lemma~\ref{le-1.1}.  
\end{proof}

\subsection{Double  complexes  for  associated  $\Delta$-manifolds}\label{ssec-3.3}

\begin{definition}
\label{def-0ma1h}
Let  $A_\bullet=(A_n)_{n\in \mathbb{N}} $  be  a  $\Delta$-manifold.  
Let  $p\in  A_n$.  
The    {\it    $\Delta$-closure}        of  $p$  in  $A_\bullet$,  
denoted  by $\Delta  p$,    is  
a  collection  of  finitely  many points  in  $A_\bullet$  given  by  
\begin{eqnarray*}
p, ~~~ \partial^n_i p,  ~~~\partial^{n-1}_j  \partial^n_i   p,  ~~~\partial^{n-2}_k \partial^{n-1}_j  \partial^n_i   p,  ~~~
\partial^{n-3}_l\partial^{n-2}_k \partial^{n-1}_j  \partial^n_i   p, ~~~\cdots
\end{eqnarray*}
for  all  possible  $i,j,k,l,\ldots$.   The  sequence  terminates  after  $n$-steps.  
\end{definition}

 \begin{lemma}\label{le-250218}
 $\Delta  p$  is  the  smallest  $\Delta$-set  containing  $p$.  
 \end{lemma}
 
 \begin{proof}
 Write  $\Delta  p=  (S_n)_{n\in  \mathbb{N}}$  where   
 \begin{eqnarray*}
 &S_n=\{p\}, \\
 &  S_{n-1}= \{\partial^n_i p\mid  0\leq  i\leq  n\},  \\
 &
 S_{n-2} =  \{\partial^{n-1}_j  \partial^n_i   p\mid    0\leq  j <  i\leq  n\},\\
  &   \cdots,\\
  & S_1=    \{ \partial^1_k  \cdots  \partial^{n-1}_j  \partial^n_i   p\mid    0\leq   k<  \cdots <  j <  i\leq  n\},
  \\
  & S_0=\{\emptyset\}.     
 \end{eqnarray*}
 By  Definition~\ref{def-7.1}~(1)  and  (2),  
 $\Delta  p$  is  a  $\Delta$-set  containing  $p$.  
  Let  $(S'_n)_{n\in \mathbb{N}}$  be  a  $\Delta$-set  containing  $p$.  
 Then  $p\in  S'_n$.  It  follows  that  $S_{m}\subseteq  S'_{m}$ 
 for  $0\leq  m\leq  n$.  
  Thus $\Delta  p$  is  the  smallest  $\Delta$-set  containing  $p$.   
 \end{proof}

\begin{definition}\label{def-0ma2h}
Let  $A_\bullet=(A_n)_{n\in \mathbb{N}} $  be  a  $\Delta$-manifold 
and  let   $B_\bullet=(B_n)_{n\in \mathbb{N}} $  be  a  graded  submanifold  of  $A_\bullet$.  
\begin{enumerate}[(1)]  
\item
We  define  the  {\it  associated  $\Delta$-manifold}  $\Delta  B_\bullet$  of  $B_\bullet$ 
to  be  the  smallest  $\Delta$-submanifold  
of  $A_\bullet$  containing  $B_\bullet$ as  a  graded  submanifold.  Precisely,  
\begin{eqnarray*}
\Delta  B_\bullet=  \bigcup_{p \in  B_\bullet }  \Delta   p;  
\end{eqnarray*}

\item
We  define the  {\it  lower-associated  $\Delta$-manifold}  $\delta  B_\bullet$  of  $B_\bullet$ 
to  be  the   largest   $\Delta$-submanifold  of  $A_\bullet$    contained  in  $B_\bullet$ as  a  graded  submanifold.   Precisely,  
\begin{eqnarray*}
\delta  B_\bullet=    \bigcup_{\Delta  p \subseteq   B_\bullet }     p.  
\end{eqnarray*}
\end{enumerate}
\end{definition}

\begin{lemma}\label{le-mboay7}
Both  $\Delta  B_\bullet$  and   $\delta  B_\bullet$
 do  not  depend  on  the  choice  of   the  ambient  $\Delta$-manifold  $A_\bullet$.   
\end{lemma}

\begin{proof}
Let  $A_\bullet$  and  $A'_\bullet$  be  $\Delta$-manifolds  such that  their  face  maps  
coincide   on  the  intersection  $A_\bullet  \cap    A' _\bullet$. 
Let  $B_\bullet$  be  a  graded  submanifold  of  $A_\bullet$  as  well as  $A'_\bullet$.  
 Then  by  Definition~\ref{def-0ma2h}  (1)  and  (2)  respectively,   
 \begin{enumerate}[(1)]
 \item
the  associated  $\Delta$-manifold  $\Delta  B_\bullet$  in  $A_\bullet$  and  
the  associated  $\Delta$-manifold  $\Delta'  B_\bullet$  in  $A'_\bullet$   are  equal;  
\item
the  lower-associated  $\Delta$-manifold  $\delta  B_\bullet$  in  $A_\bullet$  and  
the  lower-associated  $\Delta$-manifold  $\delta'  B_\bullet$  in  $A'_\bullet$   are  equal.  
\end{enumerate}
These  respectively  show  that  $\Delta  B_\bullet$  and   $\delta  B_\bullet$
 do  not  depend  on  the  choice  of   the  ambient  $\Delta$-manifold  $A_\bullet$.  
\end{proof}

\begin{lemma}\label{le-sqc}
For  any   $\Delta$-manifold  $A_\bullet=(A_n)_{n\in \mathbb{N}} $  and  any  
  graded  submanifold   $B_\bullet=(B_n)_{n\in \mathbb{N}} $    of  $A_\bullet$,    we  have  
    a  sequence  of  surjective  homomorphisms  of  double  complexes  
    \begin{eqnarray}
&&   ( \Omega^\bullet(\Delta  B_\bullet),  d,  \partial)  \longrightarrow  ({\rm  Sup}^\bullet(\Omega^\bullet(B_\bullet) ),  d,  \partial) \nonumber\\
&& \longrightarrow   
({\rm  Inf}^\bullet (\Omega^\bullet(B_\bullet)) ,  d,  \partial)
\longrightarrow    ( \Omega^\bullet(\delta  B_\bullet),  d,  \partial).      
\label{eq-sqckgoa}
    \end{eqnarray}
    Moreover,  $B_\bullet$  is  a  $\Delta$-submanifold  of  $A_\bullet$  iff   all the three  homomorphisms  in  (\ref{eq-sqckgoa})   are  the  identity. 
\end{lemma}

\begin{proof}
Substitute  $A_\bullet$  with  $\Delta  B_\bullet$   in  (\ref{eq-mkaoq1}).  
We  obtain  the first  three  double  complexes  as well as the 
first  two  surjective  homomorphisms  in  (\ref{eq-sqckgoa}).  
 The  embedding  $  \delta  B_\bullet\longrightarrow  B_\bullet$
 induces  a  surjective  pull-back  map  of  differential forms  
\begin{eqnarray}\label{eq-2.5.7a}
\Omega^\bullet( B_\bullet)\longrightarrow  \Omega^\bullet(\delta   B_\bullet).  
\end{eqnarray}
Since  
$
{\rm  Inf}^n(\Omega^\bullet(\delta  B_\bullet))= \Omega^\bullet((\delta  B)_n)
$
for  each  $n\in  \mathbb{N}$,  from   (\ref{eq-2.5.7a})  we  obtain  a  surjective  homomorphism
 of  double  complexes 
 \begin{eqnarray*}
 ({\rm  Inf}^\bullet (\Omega^\bullet(B_\bullet) ),  d,  \partial)
\longrightarrow    ( \Omega^\bullet(\delta  B_\bullet),  d,  \partial).  
 \end{eqnarray*}
We  obtain   (\ref{eq-sqckgoa}).  
Suppose $B_\bullet$  is  a  $\Delta$-submanifold  of  $A_\bullet$.  
Then  $\Delta B_\bullet = B=\delta  B_\bullet$  and  ${\rm  Inf}^\bullet (\Omega^\bullet(B_\bullet) )={\rm  Sup}^\bullet (\Omega^\bullet(B_\bullet) )$.  
Hence  all  the  homomorphisms  in   (\ref{eq-sqckgoa})  are  the  identity.  
Conversely,  suppose   all  the  homomorphisms  in   (\ref{eq-sqckgoa})  are  the  identity.  
Then  $\Delta  B=\delta   B$.   Hence  by Definition~\ref{def-0ma2h},  we  have   $\Delta B_\bullet = B_\bullet=\delta  B_\bullet$,   which  implies  $B_\bullet$  is  a  $\Delta$-submanifold  of  $A_\bullet$.  
\end{proof}

\begin{remark}
By  Lemma~\ref{le-hqlca7}  and  Lemma~\ref{le-mboay7},  the  sequence    
(\ref{eq-sqckgoa})  only  depends  on  $B_\bullet$  and  does  not depend  on  the choice  of  the  ambient  $\Delta$-manifold  $A_\bullet$.  
\end{remark}
 
  \section{Configuration  spaces  and  hypergraphs}\label{sec-conf-hpg}
  
  In  this  section,  we  review  the  definition  of  configuration  spaces  and  the  group  actions  
  of  symmetric  groups.  
  We  consider  hypergraphs  with  vertices   from  a  manifold  $M$.   Thus  the  hypergraphs  are
      subspaces  of    configuration  spaces.

Let  $M$  be  a  Riemannian  manifold  of  dimension  $m$.  
  For any  positive integer  $n$,  the  $n$-th 
  ordered  configuration  space  of  $M$  
  \begin{eqnarray}\label{eq-00bz}
  {\rm  Conf}_n(M)= \{(x_1,\ldots,x_n)\in  M^n\mid  x_i\neq  x_j {\rm  ~for~any~}i\neq  j\}
  \end{eqnarray}
   is  an  $nm$-dimensional  Riemannian  submanifold     of  the  $n$-fold  self-product  of  $M$.   
   We  call  an  element  in  ${\rm  Conf}_n(M)$  a  {\it  directed $n$-hyperedge}   on  $M$
   and  denote  it  as  ${\vec  \sigma}(M)$. 
   We  call  a  collection  of  certain    directed $n$-hyperedges  an  {\it   $n$-uniform  hyperdigraph}  on  $M$  and  denote  it  as  $\vec{\mathcal{H}}_n(M)$.  
   In  other  words,   $\vec{\mathcal{H}}_n(M)$  is  any  subspace  of  ${\rm  Conf}_n(M)$.  
   We  call  a  union    $\bigcup_{n=1}^\infty\vec{\mathcal{H}}_n(M)$  a   {\it   hyperdigraph}  on  $M$  and  
  denote  it as  $\vec{\mathcal{H}}(M)$.  
Note  that  each  $\vec{\mathcal{H}}_n(M)$  is  a  subspace  of   $ {\rm  Conf}_n(M)$  hence  $\vec{\mathcal{H}}(M)$  is  a  graded  subspace  of  the   union  
$\bigcup_{n=1}^\infty  {\rm  Conf}_n(M)$.

 The  symmetric   group  $\Sigma_n$   on  $n$-letters  acts  on  ${\rm  Conf}_n(M)$  from  left  by 
 permuting   the  order  of  $(x_1,\ldots,x_n)$.  
 The  $n$-th  unordered  configuration  space  on  $M$  is  the  orbit  space  
 \begin{eqnarray*}
  {\rm  Conf}_n(M)/\Sigma_n =   \{(x_1,\ldots,x_n)\in    {\rm  Conf}_n(M)\}  /  (x_1,\ldots,x_n)\sim (x_{s(1)},\ldots,x_{s(n)})  
 \end{eqnarray*}
 where  $s$  runs  over  $\Sigma_n$.  
  We  call  an  element  in  ${\rm  Conf}_n(M)/\Sigma_n$  a  {\it  $n$-hyperedge}   on  $M$
   and  denote  it  as  ${  \sigma}(M)$. 
   We  call  a  collection  of  certain     $n$-hyperedges  an   {\it    $n$-uniform  hypergraph}  on  $M$  and  denote  it  as  $ {\mathcal{H}}_n(M)$.  
   In  other  words,  $ {\mathcal{H}}_n(M)$   is  any  subspace  of   ${\rm  Conf}_n(M)/\Sigma_n$.   
   We  call  a  union    $\bigcup_{n=1}^\infty {\mathcal{H}}_n$  a   {\it    hypergraph}  on  $M$  and  
  denote  it as  $ {\mathcal{H}}(M)$.  
Note  that  each  ${\mathcal{H}}_n(M)$  is  a  subspace  of   $ {\rm  Conf}_n(M)/\Sigma_n$    hence  ${\mathcal{H}}(M)$  is  a  graded   subspace  of  the   union  
$\bigcup_{n=1}^\infty  {\rm  Conf}_n(M)/\Sigma_n$.     
The  action  of  $\Sigma_n$  on  ${\rm  Conf}_n(M)$  is  free  and  properly  discontinuous.  
It  induces  a  $n!$-sheeted  covering  map  
\begin{eqnarray}\label{eq-0.a1}
\pi_n:     {\rm  Conf}_n(M)\longrightarrow   {\rm  Conf}_n(M)/\Sigma_n.  
\end{eqnarray}

  \begin{example}
  Let  $M=V$  be  a  discrete  set  of  vertices.   
    Then  
    \begin{enumerate}[(1)]
    \item
  ${\rm  Conf}_n(V)$  is  the  set  of  all  the    words  $v_0\ldots  v_{n-1}$,  
 where  the  order  is  considered,     
    such  that  $v_0,\ldots,v_{n-1}\in  V$  are  distinct;  
    \item 
     ${\rm  Conf}_n(V)/\Sigma_n$  is the  set  of  all  the     sets 
    $\{v_0,\ldots,v_{n-1}\}$,  where  the  order  is  not  considered,  
  such  that  $v_0,\ldots,v_{n-1}\in  V$  are  distinct; 
  \item 
   ${\rm  Conf}_2(V)$  is  the  complete  digraph  on  $V$  where  each  edge  is  
   assigned  with double  directions,  
   \item
   ${\rm  Conf}_2(V)/\Sigma_2$  is  the  complete  graph  on  $V$; 
   \item
     $\bigcup_{n=  1}^\infty {\rm  Conf}_n(V)/\Sigma_n$  is  the  clique  complex  of  the  complete  
  graph  on  $V$;
  \item    
   a     $2$-uniform  hyperdigraph  on  $V$  is  a  digraph  on  $V$  in  the  
  usual  sense;
  \item 
  a  $2$-uniform  hypergraph  on  $V$  is  a  graph  on  $V$  in  the  usual  sense;
  \item  a     hypergraph  $\mathcal{H}$  on  $V$  is  a  hypergraph  in  the  usual  sense.   
  \end{enumerate} 
  \end{example}

  \begin{example}
  Let   $M=S^1$   be  the   unit  circle.    Then  
  \begin{enumerate}[(1)]
  \item
  ${\rm  Conf}_n(S^1)$  is  the  set  of  all  the    words  $e^{i\theta_1}\ldots  e^{i\theta_n}$,  
 where  the  order  is  considered,     
    such  that  $0\leq  \theta_1,\cdots,  \theta_n<2\pi$  are  distinct;  
    \item  
     ${\rm  Conf}_n(S^1)/\Sigma_n$  is the  set  of  all  the     sets 
    $\{e^{i\theta_1}, \ldots,  e^{i\theta_n}\}$,  where  the  order  is  not  considered,  
  such  that  $0\leq  \theta_1<\cdots< \theta_n<2\pi$  are  distinct;
  \item  
 a  directed  $2$-hyperedge  on  $S^1$  is  a  directed  secant;
 \item
   a  $2$-hyperedge  
 on  $S^1$  is  a  secant;
 \item
    a  directed  $3$-hyperedge  on  $S^1$  is  an  inscribed triangle  where  
 the  order  of  the  three  vertices  is  considered;
 \item
   a  $3$-hyperedge  
 on  $S^1$  is   an  inscribed triangle  where  
 the  order  of  the  three  vertices  is  not  considered.    
 \end{enumerate}
  \end{example}

  \begin{example}
  Let  $M=S^m$  be  the  unit  $m$-sphere  in  $\mathbb{R}^{m+1}$.   
  \begin{enumerate}[(1)]
  \item
 The  collection  of  all  the  ordered    pairs  of  antipodal  points  in  $S^m$  is  a  directed  $2$-uniform  hypergraph  $\vec{\mathcal{H}}_2(S^m)$  on  $S^m$,  which  can  be  identified  with  $S^m$;  
  \item
  The  collection  of  all  the  unordered    pairs  of  antipodal  points  in  $S^m$  is  a     $2$-uniform  hypergraph  ${\mathcal{H}}_2(S^m)$  on  $S^m$,  which  is  $\mathbb{R}P^m$;
  \item 
  The  covering  map  $\pi_2:  \vec{\mathcal{H}}_2(S^m)\longrightarrow  {\mathcal{H}}_2(S^m)$  is  
  the  canonical  $2$-sheeted  covering  map  from   $S^m$  to  $\mathbb{R}P^m$.  
  \end{enumerate}
  \end{example}

 \begin{example}\label{ex-2.9vz}
 Let  $r\geq  0$  be  a  real  number.  
 Consider  the  directed  $n$-uniform  hypergraph  
 \begin{eqnarray}\label{eq-00ba}
 {\rm  Conf}_n(M,r)= \{(x_1,\ldots,x_n)\in  M^n\mid  d(x_i,    x_j)>  2r {\rm  ~for~any~}i\neq  j\}, 
 \end{eqnarray}
 where  $d$  is  the  geodesic  distance  on  $M$,   and  the  $n$-uniform  hypergraph 
 \begin{eqnarray}\label{eq-11ba}
    {\rm  Conf}_n(M,r)/\Sigma_n=   \{(x_1,\ldots,x_n)\in     {\rm  Conf}_n(M,r)\}  /  (x_1,\ldots,x_n)\sim (x_{s(1)},\ldots,x_{s(n)})     
 \end{eqnarray}
  where  $s$  runs  over  $\Sigma_n$.  
  Note  that  if  $M$  is  without  boundary,  then  both  (\ref{eq-00ba})  and  (\ref{eq-11ba})  are  open  manifolds  without  boundary  while  
  \begin{eqnarray*}
  \bigcap_{r<  s}   {\rm  Conf}_n(M,r)=\{(x_1,\ldots,x_n)\in  M^n\mid  d(x_i,    x_j)\geq  2s {\rm  ~for~any~}i\neq  j\}
  \end{eqnarray*}
  and  
   \begin{eqnarray*}
   \bigcap_{r<  s} {\rm  Conf}_n(M,r)/\Sigma_n        
 \end{eqnarray*}
 could  be    manifolds     with  boundaries.  
   We  have     filtrations    
  \begin{eqnarray}\label{eq-099a}
  \{  {\rm  Conf}_n(M,r)\mid   r\geq 0\},  ~~~~~~   \{  {\rm  Conf}_n(M,r)/\Sigma_n\mid   r\geq 0\}
  \end{eqnarray}
  such that  
  \begin{enumerate}[(1)]
  \item
    $ {\rm  Conf}_n(M,0)= {\rm  Conf}_n(M)$  and  $ {\rm  Conf}_n(M,0)/\Sigma_n= {\rm  Conf}_n(M)/\Sigma_n$;
    \item   
  $ {\rm  Conf}_n(M,r) \supseteq   {\rm  Conf}_n(M,s)$   and 
    $ {\rm  Conf}_n(M,r) /\Sigma_n\supseteq {\rm  Conf}_n(M,s)/\Sigma_n$    for  any   $0\leq  r  <s$; 
  \item
  if  $M$  is  path-connected,  then   
  \begin{eqnarray*}
 & \lim_{r\to+\infty}    {\rm  Conf}_n(M,r)  =  \bigcap_{r\geq  0}  {\rm  Conf}_n(M,r)  =\emptyset,\\
 & \lim_{r\to+\infty}     {\rm  Conf}_n(M,r)/\Sigma_n  =  \bigcap_{r\geq  0}  {\rm  Conf}_n(M,r)/\Sigma_n  =\emptyset;
  \end{eqnarray*}
    \item
    the  $\Sigma_n$-action  on   $ {\rm  Conf}_n(M,r)$  is  equivariant,  i.e.    for  any   $0\leq  r  <s$,  the  diagram   commutes  
    \begin{eqnarray*}
   & \xymatrix{
     {\rm  Conf}_n(M,s) \ar[r]\ar[d]_{\Sigma_n}  & {\rm  Conf}_n(M,r)\ar[d]^{\Sigma_n}\\
         {\rm  Conf}_n(M,s)/\Sigma_n \ar[r]   & {\rm  Conf}_n(M,r)/\Sigma_n 
    }
    \end{eqnarray*}
    which  implies for  any   $0\leq  r  <s$,  the  diagram   commutes  
\begin{eqnarray*}
  &    \xymatrix{
     {\rm  Conf}_n(M,s) \ar[r]\ar[d]_{\pi_n}  & {\rm  Conf}_n(M,r)\ar[d]^{\pi_n}\\
         {\rm  Conf}_n(M,s)/\Sigma_n \ar[r]   & {\rm  Conf}_n(M,r)/\Sigma_n. 
    }
    \end{eqnarray*}
    Here  in  the  last two  diagrams,  the  horizontal  maps  are   canonical  inclusions.  
 \end{enumerate}
 In  particular,  
 \begin{enumerate}[(i)]
 \item
 Let  $M=V$  be  a  discrete  set.  
Then for  any  $p,q\in  V$,  we  have   $d(p,q)=0$  if  $p=q$  and  $d(p,q)=+\infty$  if  $p\neq q$. 
The  filtrations  (\ref{eq-099a})   are  trivial such  that   both    $ {\rm  Conf}_n(V,r)$  
and  $  {\rm  Conf}_n(V,r)/\Sigma_n$  do  not  depend  on  $r$  for  all  real  numbers  $r\geq  0$.   

 \item
 Let  $M=S^1$  be  the  unit  sphere.  Then   
 $  {\rm  Conf}_n(S^1,r)/\Sigma_n$  is   the  collection   of   all    the     sets 
    $\{e^{i\theta_1}, \ldots,  e^{i\theta_n}\}$,  where  the  order  is  not  considered,  
  such  that  $0\leq  \theta_1<\cdots  <  \theta_n<2\pi$  and   $\theta_{i+1}-\theta_i>2r$ for  
  any  $i=1,\ldots, n-1$.  
 In  particular,    $  {\rm  Conf}_n(S^1,r)/\Sigma_n=\emptyset$  if  $r\geq    \pi / n $.

 \item
 Let  $M=S^m$  be the  unit  $m$-sphere.  
 Then   the  directed  $n$-uniform  hypergraph  
 \begin{eqnarray*}
 \bigcap_{r<  \pi/2}{\rm  Conf}_n(S^m,r)= \{(x_1,\ldots,x_n)\in  \prod_n  S^m \mid  d(x_i,    x_j)\geq   \pi {\rm  ~for~any~}i\neq  j\}   
 \end{eqnarray*}
   is  the  collection  of  all  the  ordered  pairs  of  antipodal  points  in  $S^m$   if   $n=2$   (note  that  
  $d(p,q)\leq  \pi$  for  any  $p,q\in  S^m$  and  the  equality  holds  iff  $p,q$ are  antipodal)  and  is  $\emptyset$  if  $n\geq 3$.      The  $n$-uniform  hypergraph   
    \begin{eqnarray*}
 \bigcap_{r<  \pi/2}{\rm  Conf}_n(S^m,r)/ \Sigma_n  
 \end{eqnarray*}
 is   $\mathbb{R}P^m$  if  $n=2$  and  is  $\emptyset$  if  $n\geq  3$.  
 \end{enumerate}
 \end{example}

\section{Morphisms  of    configuration  spaces  and  hypergraphs }\label{sec5}

In  this section,  we  consider  morphisms  of  configuration  spaces  and  hypergraphs.  
We  construct  some  automorphism  groups  for  configuration  spaces  and  hypergraphs.

 Let $\varphi:  M\longrightarrow  M$ be   
      any    map  which  may  not  be  assumed   continuous.     
      Then  $\varphi$    induces   a    map 
          \begin{eqnarray}\label{eq-1.00a2}
    {\rm  Conf}(\varphi)/\Sigma_\bullet:  \bigcup_{n=1}^\infty  {\rm  Conf}_n(M)/\Sigma_n \longrightarrow  \bigcup_{n=1}^\infty  {\rm  Conf}_n(M)/\Sigma_n
    \end{eqnarray}
   sending   $\{x_1,\ldots,x_n\}\in   {\rm  Conf}_n(M)/\Sigma_n$
   to   $\{\varphi(x_1),\ldots,\varphi(x_n)\}\in   {\rm  Conf}_k(M)/\Sigma_k$,  
   where   $1\leq  k\leq  n$  is  the  number  of   distinct  points  in  $\varphi(x_1),\ldots,\varphi(x_n)$.  
    We  regard    
    the  configuration  spaces 
    ${\rm  Conf}_k(M)/\Sigma_k$  and  ${\rm  Conf}_{l}(M)/\Sigma_l$   to  be   disjoint  
    for  $k\neq  l$.   Note  that  even  if  $\varphi$  is  continuous,  ${\rm  Conf}(\varphi)/\Sigma_\bullet$  may  not  be  continuous.   
       Let  $\mathcal{H}(M)$   and  $\mathcal{H}'(M)$   be  two  hypergraphs  on  $M$ 
    such  that  the  restriction  of   ${\rm  Conf}(\varphi)/\Sigma_\bullet$  to  $\mathcal{H}(M)$  induces  a  map  
    \begin{eqnarray}\label{eq-1.00a1}
    {\rm  Conf}(\varphi)/\Sigma_\bullet:  \mathcal{H}(M)\longrightarrow  \mathcal{H}'(M). 
    \end{eqnarray}
    We  call  (\ref{eq-1.00a1})  a  {\it  morphism}  of  hypergraphs  on  $M$  induced  by  $\varphi$ 
    and  simply  abbreviate   (\ref{eq-1.00a1})   as  
    \begin{eqnarray}\label{eq-1.00a90}
    \varphi: \mathcal{H}(M)\longrightarrow  \mathcal{H}'(M) 
    \end{eqnarray}
      if there is no ambiguity.

      Suppose  in  addition  that  
     $\varphi:  M\longrightarrow  M$  is  injective.  
     Then  (\ref{eq-1.00a2})  is  a  graded  injective  map  sending  each  
    ${\rm  Conf}_n(M)/\Sigma_n$  into  ${\rm  Conf}_n(M)/\Sigma_n$.
     Thus  (\ref{eq-1.00a1})   or  (\ref{eq-1.00a90})   is   a  graded  injective  map  sending  each  
    $\mathcal{H}_n(M)$  into  $\mathcal{H}'_n(M)$.  
        Moreover,  $\varphi$  induces  a   graded  injective  map
    \begin{eqnarray}\label{eq-99.7ab}
    {\rm  Conf}(\varphi):    \bigcup_{n=1}^\infty  {\rm  Conf}_n(M) \longrightarrow  \bigcup_{n=1}^\infty  {\rm  Conf}_n(M) 
    \end{eqnarray}
    sending   $ (x_1,\ldots,x_n)\in   {\rm  Conf}_n(M)$
   to   $(\varphi(x_1),\ldots,\varphi(x_n))\in   {\rm  Conf}_n(M)$.  
   The  following  diagram  commutes  
   \begin{eqnarray*}
   \xymatrix{
   {\rm  Conf}_n(M)  \ar[rr]  ^-{  {\rm  Conf}(\varphi)}  \ar[d]_{\pi_n}  && {\rm  Conf}_n(M)  \ar[d]^{\pi_n} \\
    {\rm  Conf}_n(M)/\Sigma_n \ar[rr]^-{{\rm  Conf}(\varphi)/\Sigma_n } && {\rm  Conf}_n(M) /\Sigma_n. 
   }
   \end{eqnarray*}
    Let  $\vec{\mathcal{H}}(M)$   and  $\vec{\mathcal{H}}'(M)$   be  two  hyperdigraphs  on  $M$ 
    such  that  the  restriction  of   ${\rm  Conf}(\varphi)$  to  $\vec{\mathcal{H}}(M)$  induces  a  map  
    \begin{eqnarray}\label{eq-27.00a1}
    {\rm  Conf}(\varphi):  \vec{\mathcal{H}}(M)\longrightarrow  \vec{\mathcal{H}}'(M). 
    \end{eqnarray}
   Then  (\ref{eq-27.00a1})  is  injective.  
     We  call  (\ref{eq-27.00a1})  a  {\it  morphism}  of  hyperdigraphs  on  $M$  induced  by  $\varphi$ 
    and  simply  abbreviate   (\ref{eq-27.00a1})   as  
    \begin{eqnarray}\label{eq-27.00a90}
    \varphi: \vec{\mathcal{H}}(M)\longrightarrow \vec{ \mathcal{H}}'(M) 
    \end{eqnarray}
      if there is no ambiguity.  
   Suppose  $\mathcal{H}_n(M)=\pi_n(\vec{\mathcal{H}}_n(M))$  and   $\mathcal{H}'_n(M)=\pi_n(\vec{\mathcal{H}}'_n(M))$   for  each  $n\in \mathbb{N}$.  
   Then 
      \begin{eqnarray*}
   \vec{\mathcal{H}}_n(M)\subseteq  \pi_n^{-1}(\mathcal{H}_n(M)),~~~~~~
   \vec{\mathcal{H}'}_n(M)\subseteq  \pi_n^{-1}(\mathcal{H}'_n(M))   
   \end{eqnarray*}
   and  
    the  following  diagram  commutes  
   \begin{eqnarray*}
   \xymatrix{
 \vec{\mathcal{H}}_n(M)    \ar[r]  ^{  \varphi}  \ar[d]_{\pi_n}  & \vec{\mathcal{H}}'_n(M)  \ar[d]^{\pi_n} \\
   \mathcal{H}_n(M)   \ar[r]^{\varphi } & \mathcal{H}'_n(M).       
   }
   \end{eqnarray*}
  Furthermore,    if  $\varphi$  is  continuous  (resp.  smooth),  then  all  the  four  maps  
         (\ref{eq-1.00a2}),     (\ref{eq-1.00a90}),  (\ref{eq-99.7ab})  and  (\ref{eq-27.00a90})    are   continuous (resp.  smooth).

The  homeomorphism  group  ${\rm  Homeo}(M)$  is  the   group  
consisting  of  all  the  homeomorphisms  $ \varphi$  from  $M$  to itself.  
For  any  $\varphi\in  {\rm  Homeo}(M)$,  
the induced  maps   
 (\ref{eq-1.00a2})  
 and    
 (\ref{eq-99.7ab})   are       homeomorphisms.  
 Consequently,  we  have  induced   group  actions  of  ${\rm  Homeo}(M)$   on  
both  $ \bigcup_{n=1}^\infty  {\rm  Conf}_n(M)/\Sigma_n$ 
and        $ \bigcup_{n=1}^\infty  {\rm  Conf}_n(M)$.  
  A  homeomorphism  $\varphi\in {\rm  Homeo}(M)$  is  an  {\it  isometry}  if 
   $d(p,q)=  d(\varphi(p),\varphi(q))$   for  any  
   $p,q\in  M$,  where  $d$  is  the  geodesic  distance  on  $M$.  
The  isometry group  ${\rm  Isom}(M)$  is  the group
 consisting  of  all the  isometries  from    
   $M$  to  itself,  which  is  a  subgroup  of  ${\rm  Homeo}(M)$.

For  any  hypergraph  $\mathcal{H}(M)$  on  $M$, 
we  define  
the  homeomorphism  group  ${\rm  Homeo}(\mathcal{H}(M))$
to be   the  subgroup  of    ${\rm  Homeo}(M)$   consisting   of   $\varphi\in {\rm  Homeo}(M)$ 
 such that  $\varphi$  induces  a  morphism  from  $\mathcal{H}(M)$  to  itself   by  (\ref{eq-1.00a1})
    and   define  
   the  isometry  group  ${\rm  Isom}(\mathcal{H}(M))$
to  be  the  subgroup  of   ${\rm  Isom}(M)$  consisting  of   
 $\varphi\in  {\rm  Isom}(M)$   such that  $\varphi$  induces  a  morphism  from  $\mathcal{H}(M)$  to  itself   by  (\ref{eq-1.00a1}).  
 For  any  hyperdigraph  $\vec{\mathcal{H}}(M)$  on  $M$, 
  we  define  
the  homeomorphism  group  ${\rm  Homeo}(\vec{\mathcal{H}}(M))$
to be   the  subgroup  of    ${\rm  Homeo}(M)$   consisting   of   $\varphi\in {\rm  Homeo}(M)$ 
 such that  $\varphi$  induces  a  morphism  from  $\vec{\mathcal{H}}(M)$  to  itself   by  (\ref{eq-99.7ab}) 
   and   define  
   the  isometry  group  ${\rm  Isom}(\vec{\mathcal{H}}(M))$
to  be  the  subgroup  of   ${\rm  Isom}(M)$  consisting  of   
 $\varphi\in  {\rm  Isom}(M)$   such that  $\varphi$  induces  a  morphism  from  $\vec{\mathcal{H}}(M)$  to  itself   by  (\ref{eq-99.7ab}).  
  Then     we  have  a   group  action  of  ${\rm  Homeo}({\mathcal{H}}(M))$   on  
 ${\mathcal{H}}(M)$   and  a   group  action  of  ${\rm  Homeo}(\vec{\mathcal{H}}(M))$   on  
 $\vec{\mathcal{H}}(M)$.  
Note  that   ${\rm  Isom}({\mathcal{H}}(M))$  is  a  subgroup  of     ${\rm  Homeo}({\mathcal{H}}(M))$   and   
${\rm  Isom}(\vec{\mathcal{H}}(M))$  is  a  subgroup  of     ${\rm  Homeo}(\vec{\mathcal{H}}(M))$.  
  Moreover,  
 if  ${\mathcal{H}}(M)=\pi_\bullet(\vec{\mathcal{H}}(M))$, 
   then  $ {\rm  Homeo}(\vec{\mathcal{H}}(M))$   is   a   normal  
   subgroup  of   $ {\rm  Homeo}({\mathcal{H}}(M))$  
   and   $ {\rm  Isom}(\vec{\mathcal{H}}(M))$   is   a   normal  
   subgroup  of   $ {\rm  Isom}({\mathcal{H}}(M))$. 

    Let  $\sigma\in\mathcal{H}(M)$.  We  define  the  {\it  stabilizer}  of  $\sigma$,  
denote  as  ${\rm  Stab}(\sigma)$,  
  to  be  the   normal   subgroup  of   $ {\rm  Homeo}(M) $ 
 consisting  of  all   the  homeomorphisms  $\varphi\in {\rm  Homeo}(M) $   such  that  
 $\varphi(\sigma)=\sigma$.  
  We  define  the  {\it  stabilizer}  of  $\mathcal{H}(M)$,   
  denote  as   ${\rm  Stab}(\mathcal{H}(M))$,  
 to  be  the   normal  subgroup  of   $ {\rm  Homeo}(M) $  given  by  the  intersection  
   \begin{eqnarray*}
   {\rm  Stab}(\mathcal{H}(M)) =  \bigcap  _{\sigma\in\mathcal{H}(M) }  {\rm  Stab}(\sigma).  
   \end{eqnarray*}
   Then  $ {\rm  Stab}(\mathcal{H}(M))$  is   the  normal  subgroup  
   of  $ {\rm  Homeo}(\mathcal{H}(M)) $   consisting   of  $\varphi $ 
   such  that  $\varphi(\sigma)=\sigma$  for  any  $\sigma\in \mathcal{H}(M)$.

   Similarly,  let  $\vec{\sigma}\in\vec{\mathcal{H}}(M)$.  
   We  define  the  {\it  stabilizer}  of  $\vec\sigma$,  
denote  as  ${\rm  Stab}(\vec\sigma)$,  
  to  be  the   normal  subgroup  of   $ {\rm  Homeo}(M) $ 
 consisting  of  all   the  homeomorphisms  $\varphi\in {\rm  Homeo}(M) $   such  that  
 $\varphi(\vec\sigma)=\vec\sigma$.  
 Note  that  if  $\vec\sigma=(x_1,\ldots,x_n)$,   then  
 \begin{eqnarray*}
 {\rm  Stab}(\vec\sigma) =\bigcap_{i=1}^n {\rm  Stab}(x_i) 
 \end{eqnarray*}
 where  ${\rm  Stab}(x_i) $  is  the  group  of  all    $\varphi\in {\rm  Homeo}(M) $   such  that  
 $\varphi(x_i)=x_i$.  
 Hence  
 $ {\rm  Stab}(\vec\sigma)$  is  a  normal   subgroup  of  
  ${\rm  Stab}(\sigma)$
   where  $\sigma=\pi_\bullet(\vec\sigma)$.  
  We  define  the  {\it  stabilizer}  of  $\vec{\mathcal{H}}(M)$,   
  denote  as   ${\rm  Stab}(\vec{\mathcal{H}}(M))$,  
 to  be  the   normal   subgroup  of   $ {\rm  Homeo}(M) $  given  by  the  intersection  
   \begin{eqnarray*}
   {\rm  Stab}(\vec{\mathcal{H}}(M)) =  \bigcap  _{\vec\sigma\in\vec{\mathcal{H}}(M) }  {\rm  Stab}(\vec\sigma).  
   \end{eqnarray*}
   Then  $ {\rm  Stab}(\vec{\mathcal{H}}(M))$  is   the  normal   subgroup  
   of  $ {\rm  Homeo}(\vec{\mathcal{H}}(M)) $   consisting   of  $\varphi$ 
   such  that  $\varphi(\vec\sigma)=\vec\sigma$  for  any  $\vec\sigma\in \vec{\mathcal{H}}(M)$.  
   Note  that  if  ${\mathcal{H}}(M)=\pi_\bullet(\vec{\mathcal{H}}(M))$, 
   then  $ {\rm  Stab}(\vec{\mathcal{H}}(M))$   is   a   normal  
   subgroup  of   $ {\rm  Stab}({\mathcal{H}}(M))$. 
   
   \begin{lemma}
   \label{pr-3-vq1}
   For  any  hyperdigraph  $\vec{\mathcal{H}}(M)$    and  any  hypergraph $\mathcal{H}(M)$
   on  $M$,   
   \begin{enumerate}[(1)]
   \item
    the  intersection  $  {\rm  Stab}(\mathcal{H}(M))\cap  {\rm  Isom}(\mathcal{H}(M))$  is  a  
    normal  subgroup  of  ${\rm  Isom}(\mathcal{H}(M))$; 
     \item
    the  intersection  $  {\rm  Stab}(\vec{\mathcal{H}}(M))\cap  {\rm  Isom}(\vec{\mathcal{H}}(M))$  is  a  
    normal  subgroup  of  ${\rm  Isom}(\vec{\mathcal{H}}(M))$;
    \item   
    if     $\pi_\bullet(\vec{\mathcal{H}}(M))=\mathcal{H}(M)$,  then   $  {\rm  Stab}(\vec{\mathcal{H}}(M))\cap  {\rm  Isom}(\vec{\mathcal{H}}(M))$  
    is  a  normal  subgroup  of  $  {\rm  Stab}(\mathcal{H}(M))\cap  {\rm  Isom}(\mathcal{H}(M))$.  
    \end{enumerate}  
   \end{lemma}
   
   \begin{proof}
It  is  direct  to  verify  that    $  {\rm  Stab}(\mathcal{H}(M))$ 
   is  a  normal  subgroup  of  ${\rm  Homeo}(\mathcal{H}(M))$  
   and  $  {\rm  Stab}(\vec{\mathcal{H}}(M))$ 
   is  a  normal  subgroup  of  ${\rm  Homeo}(\vec{\mathcal{H}}(M))$.    Thus  by   the 
      isomorphism  theorem,  we  obtain  (1)  and  (2).

 Suppose 
 ${\mathcal{H}}(M)=\pi_\bullet(\vec{\mathcal{H}}(M))$.   
   It  is  direct  to  verify  that  $ {\rm  Stab}(\vec{\mathcal{H}}(M))$   is   a   normal  
   subgroup  of   $ {\rm  Stab}({\mathcal{H}}(M))$.   
   Note  that 
   \begin{eqnarray*}
     {\rm  Stab}(\vec{\mathcal{H}}(M))\cap  {\rm  Isom}(\vec{\mathcal{H}}(M))&=& {\rm  Stab}(\vec{\mathcal{H}}(M))\cap  {\rm  Isom}(M),\\
          {\rm  Stab}({\mathcal{H}}(M))\cap  {\rm  Isom}({\mathcal{H}}(M))&=& {\rm  Stab}({\mathcal{H}}(M))\cap  {\rm  Isom}(M).  
     \end{eqnarray*}  
 Since ${\rm  Stab}(\vec{\mathcal{H}}(M))\cap  {\rm  Isom}(M)$  
 is  a  normal  subgroup  of  ${\rm  Stab}({\mathcal{H}}(M))\cap  {\rm  Isom}(M)$,  
 we  obtain  (3).  
   \end{proof}

   We  define  the {\it  automorphism group}  of  
   $\mathcal{H}(M)$  to  be  the  quotient  group    
   \begin{eqnarray*}
   {\rm  Aut}(\mathcal{H}(M))={\rm  Homeo}(\mathcal{H}(M))/{\rm  Stab}(\mathcal{H}(M))   
   \end{eqnarray*}
   and   with the  help  of    Lemma~\ref{pr-3-vq1}~(1)   define  the  {\it  isometric  automorphism group}  of  
   $\mathcal{H}(M)$  to  be  the  quotient  group    
   \begin{eqnarray*}
   {\rm   Aut}_{\rm  Isom}(\mathcal{H}(M))={\rm  Isom}(\mathcal{H}(M))/
   ({\rm  Stab}(\mathcal{H}(M))\cap   {\rm  Isom}(\mathcal{H}(M))).  
   \end{eqnarray*}
  By  the 
     isomorphism  theorem, $ {\rm   Aut}_{\rm  Isom}(\mathcal{H}(M))$  is  a  normal  
   subgroup  
   of  $ {\rm  Aut}(\mathcal{H}(M))$.   
   Similarly,   We  define  the {\it  automorphism group}  of  
   $\vec{\mathcal{H}}(M)$  to  be  the  quotient  group    
   \begin{eqnarray*}
   {\rm  Aut}(\vec{\mathcal{H}}(M))={\rm  Homeo}(\vec{\mathcal{H}}(M))/{\rm  Stab}(\vec{\mathcal{H}}(M))   
   \end{eqnarray*}
   and   with the  help  of    Lemma~\ref{pr-3-vq1}~(2)   define  the  {\it  isometric  automorphism group}  of  
   $\vec{\mathcal{H}}(M)$  to  be  the  quotient  group    
   \begin{eqnarray*}
   {\rm   Aut}_{\rm  Isom}(\vec{\mathcal{H}}(M))={\rm  Isom}(\vec{\mathcal{H}}(M))/
   ({\rm  Stab}(\vec{\mathcal{H}}(M))\cap   {\rm  Isom}(\vec{\mathcal{H}}(M))).  
   \end{eqnarray*}
  By  the 
     isomorphism  theorem, $ {\rm   Aut}_{\rm  Isom}(\vec{\mathcal{H}}(M))$  is  a  normal  
   subgroup  
   of  $ {\rm  Aut}(\vec{\mathcal{H}}(M))$.   
   
   \begin{proposition}\label{pr-3.maq}
   \begin{enumerate}[(1)]
   \item
   The  group  action  of   $   {\rm  Aut}(\mathcal{H}(M))$     on  $\mathcal{H}(M)$  is   faithful;
   \item
  The  group  action  of   $   {\rm  Aut}(\vec{\mathcal{H}}(M))$     on  $\vec{\mathcal{H}}(M)$  is   faithful;

   \item
    if     $\pi_\bullet(\vec{\mathcal{H}}(M))=\mathcal{H}(M)$,  then   
    for  any  $\varphi\in  {\rm  Aut}(\vec{\mathcal{H}}(M))$  there  exists  
    a  unique  $\pi_\bullet(\varphi)\in  {\rm  Aut}(\mathcal{H}(M))$   
    such  that  the  diagram  commutes 
    \begin{eqnarray*}
    \xymatrix{
    \vec{\mathcal{H}}(M)\ar[r]^-{\varphi}\ar[d]_-{\pi_\bullet} & \vec{\mathcal{H}}(M)\ar[d]^-{\pi_\bullet}\\
    \mathcal{H}(M)\ar[r]^-{\pi_\bullet(\varphi)} & \mathcal{H}(M). 
    }
    \end{eqnarray*}
    The  induced  homomorphism  
    \begin{eqnarray}\label{eq-oa7cp1}
    \pi_\bullet:  {\rm  Aut}(\vec{\mathcal{H}}(M))
    \longrightarrow  {\rm  Aut}({\mathcal{H}}(M)) 
    \end{eqnarray}
      sending  $\varphi$  to  $\pi_\bullet(\varphi)$ 
     is  surjective.  
   \end{enumerate}
   \end{proposition}
   
   \begin{proof}
   (1)  
   Let  $\varphi_1,\varphi_2\in   {\rm  Aut}(\mathcal{H}(M))$  such  that  
   $\varphi_1\neq  \varphi_2$.  
  To  prove  (1),   it  suffices  to  prove there  exists  $\sigma\in  \mathcal{H}(M)$  such  that  
   $\varphi_1(\sigma)\neq  \varphi_2(\sigma)$.  
   Suppose  to the  contrary,     $\varphi_1(\sigma)=  \varphi_2(\sigma)$
   for  any  $\sigma\in  \mathcal{H}(M)$.  
   Then  $\varphi_2^{-1}\varphi_1(\sigma)=\sigma$   for  any  $\sigma\in  \mathcal{H}(M)$.  
   It  follows  that  $\varphi_2^{-1}\varphi_1\in  {\rm  Stab}(\mathcal{H}(M))$,  
   which  implies  $\varphi_1=\varphi_2$  in  ${\rm  Aut}(\mathcal{H}(M))$.  
   We get a  contradiction.

   (2)   The  proof  of  (2)  is  similar  with  (1).  The  difference  is  to   substitute  the  hyperedges  with  directed  hyperedges.

   (3)  Suppose     $\pi_\bullet(\vec{\mathcal{H}}(M))=\mathcal{H}(M)$.    
   Let  $\varphi:  M\longrightarrow  M$   be  a  representative  of  an  element 
     in  $   {\rm  Aut}(\vec{\mathcal{H}}(M))$. 
   Then   
    \begin{eqnarray}\label{eq-mma16}
   \varphi: (x_1,\ldots,x_n)  \mapsto (\varphi(x_1),\ldots,\varphi(x_n))   
   \end{eqnarray}
      gives  a  homeomorphism  from $\vec{\mathcal{H}}(M)$  to  itself.  
   Let  
     \begin{eqnarray}\label{eq-mma17}
   \pi_\bullet(\varphi): \{x_1,\ldots,x_n\}  \mapsto \{\varphi(x_1),\ldots,\varphi(x_n)\}.  
   \end{eqnarray}
   Then  
   $\pi_\bullet (\varphi)$  gives a  homeomorphism  from 
   $  {\mathcal{H}}(M)$  
   to  itself. 
   It  is  clear  that  the  element (\ref{eq-mma17})  in   ${\rm  Aut}(\mathcal{H}(M))$  
    does  not  depend  on 
   the  choice  of  the representative  $\varphi:  M\longrightarrow  M$  of 
   the  element  (\ref{eq-mma16}) in   $   {\rm  Aut}(\vec{\mathcal{H}}(M))$.  
    Thus   $\pi_\bullet(\varphi)\in  {\rm  Aut}(\mathcal{H}(M))$.   
    Hence (\ref{eq-oa7cp1})  is  well-defined.  
    Let   $\psi:  M\longrightarrow  M$  be  a  
    representative  of  an  element  in  ${\rm  Aut}(\mathcal{H}(M))$.   
    Then  
        \begin{eqnarray*} 
   \psi : \{x_1,\ldots,x_n\}  \mapsto \{\psi(x_1),\ldots,\psi(x_n)\}  
   \end{eqnarray*}
gives a  homeomorphism  from 
   $  {\mathcal{H}}(M)$  
   to  itself. 
   Let  
   \begin{eqnarray*}
   \tilde\psi:  (x_1,\ldots,x_n)  \mapsto (\psi(x_1),\ldots,\psi(x_n)). 
   \end{eqnarray*}
   Then  $\tilde \psi$  gives  a  homeomorphism from $\vec{\mathcal{H}}(M)$  to  itself.   
 Thus  $\tilde\psi$  represents  an  element  in   $   {\rm  Aut}(\vec{\mathcal{H}}(M))$
 such  that  $\pi_\bullet(\tilde \psi)=\psi$.  
 Therefore,  (\ref{eq-oa7cp1})  is  surjective.  
   \end{proof}
   
   \begin{example}
Let  $M=V$  be  a  discrete  set.  
Then  
\begin{eqnarray}\label{eq-2.01e}
{\rm  Homeo}(V)= {\rm  Isom}(V)=\{{\rm  bijections~}f:  V\longrightarrow  V   \}.  
\end{eqnarray}
In  addition,  if  $V$  is  finite, then  (\ref{eq-2.01e})  is  the  symmetric  group  $\Sigma_{|V|}$.  
Particularly,  take  $V= \{v_0,v_1,v_2,v_3\}$  as  an  example.  
\begin{enumerate}[(1)]
\item
Let  $\mathcal{H}=\{  \{v_1,v_2\},  \{v_0,v_2\}, \{v_0,v_1\},  \{v_0,v_1,v_2\}\}$.   
Then  
$
{\rm  Homeo}(\mathcal{H})= \Sigma_3
$
is  the  symmetric  group  on  $v_0,v_1,v_2$,  
$
{\rm  Stab}(\mathcal{H})= 1
$ 
is  the  identity  map  on  $V$,  and  consequently   ${\rm  Aut}(\mathcal{H})=\Sigma_3$ 
is  the  symmetric group  on   $\sigma_0=\{v_1,v_2\}$,   
$\sigma_1=\{v_0,v_2\}$  and  $\sigma_2=\{v_0,v_1\}$;   
\item
Let  
$\mathcal{H}= \{\{v_0,v_1\}, \{v_2,v_3\}\}$.   
Then  $
{\rm  Homeo}(\mathcal{H})
$  is  the   dihedral  group   of  order  $8$:   
\begin{eqnarray*}
&&\alpha= 
 {{v_0,v_1,v_2,v_3}\choose{v_3,v_2,v_0,v_1}},  ~~~~~~   \alpha^2= {{v_0,v_1,v_2,v_3}\choose{v_1,v_0,v_3,v_2}},~~~~~~\alpha^3={{v_0,v_1,v_2,v_3}\choose{v_2,v_3,v_1,v_0}},\\ 
 &&\beta=  {{v_0,v_1,v_2,v_3}\choose{v_1,v_0,v_2,v_3}},  ~~~~~~\alpha^2\beta=\beta\alpha^2={{v_0,v_1,v_2,v_3}\choose{v_0,v_1,v_3,v_2}},\\
 &&
 \alpha\beta=\beta\alpha^3={{v_0,v_1,v_2,v_3}\choose{v_2,v_3,v_0,v_1}},
 ~~~~~~\beta\alpha= \alpha^3\beta=  {{v_0,v_1,v_2,v_3}\choose{v_3,v_2,v_1,v_0}} 
\end{eqnarray*} 
where  $\alpha^4=\beta^2=1$  and   $\alpha\beta\alpha=\beta$.    
 Moreover,   ${\rm  Stab}(\mathcal{H})=\mathbb{Z}_2\oplus\mathbb{Z}_2
$ 
is  generated  by  the  two  permutations  $\alpha^2$  and $\beta$,  consisting  of  the  four  elements 
$1$,  $\alpha^2$,  $\beta$,  $\alpha^2\beta=\beta\alpha^2$.  
It  is  direct that  ${\rm  Stab}(\mathcal{H})$  is  a  normal  subgroup  of  ${\rm  Homeo}(\mathcal{H})$.  
 Consequently,   
 ${\rm  Aut}(\mathcal{H})=\mathbb{Z}_2$  is  generated  by  a   permutation  
 \begin{eqnarray*}
 {\sigma_1,\sigma_2}\choose{\sigma_2,\sigma_1} 
 \end{eqnarray*}
   of  the  hyperedges  where  $\sigma_1= \{v_2,v_3\}$ and 
 $\sigma_2=\{v_0,v_1\}$;   

\item
Let  $\mathcal{H}= \{\{v_0,v_1\}\}$.   Then  
\begin{eqnarray*}
{\rm  Homeo}(\mathcal{H})={\rm  Stab}(\mathcal{H}) = \mathbb{Z}_2\oplus \mathbb{Z}_2  
\end{eqnarray*}
  is  generated  by 
\begin{eqnarray*}
&&\alpha = {{v_0,v_1,v_2,v_3}\choose{v_1,v_0,v_2,v_3}},~~~~~~  \beta={{v_0,v_1,v_2,v_3}\choose{v_0,v_1,v_3,v_2}}  
\end{eqnarray*}
such  that  $\alpha^2=\beta^2=1$.   Consequently,   
 ${\rm  Aut}(\mathcal{H})=1$;
 \item
 Let  $\mathcal{H}=\{\{v_0,v_1,v_2\},  \{v_1,v_2,v_3\}\}$.   Then  
      Then  
\begin{eqnarray*}
{\rm  Homeo}(\mathcal{H})  =\mathbb{Z}_2\oplus \mathbb{Z}_2 
\end{eqnarray*}
is  generated  by  
\begin{eqnarray*}
&&\alpha = {{v_0,v_1,v_2,v_3}\choose{v_3,v_1,v_2,v_0}},~~~~~~  \beta={{v_0,v_1,v_2,v_3}\choose{v_0,v_2,v_1,v_3}}  
\end{eqnarray*}
and 
 \begin{eqnarray*}
 {\rm  Stab}(\mathcal{H}) = \mathbb{Z}_2  
\end{eqnarray*}
  is  generated  by  $\beta$.  
Thus  ${\rm  Aut}(\mathcal{H})=\mathbb{Z}_2$  is  generated  by  $\alpha$.    
\end{enumerate}
\end{example}

\begin{example}\label{ex-1-sec5}
  Consider  
   the  directed  $n$-uniform   hypergraphs
  \begin{eqnarray}\label{eq-ex0779a}
  {\rm  Conf}_n(M,r),   ~~~~~~ \bigcap_{r\in \mathbb{I}}{\rm  Conf}_n(M,r) 
\end{eqnarray}
and  
  the  $n$-uniform   hypergraphs  
\begin{eqnarray}\label{eq-ex03a}
  {\rm  Conf}_n(M,r)/\Sigma_n,   ~~~~~~ \bigcap_{r\in \mathbb{I}}{\rm  Conf}_n(M,r)/\Sigma_n  
\end{eqnarray}
on  $M$     for  any  fixed  $r\geq  0$  and      any  $\mathbb{I} \subseteq  \mathbb{R}$. 
For  any  $\varphi\in  {\rm  Isom}(M)$,  by  the    filtrations
  in   (\ref{eq-099a})  and      (\ref{eq-1.00a2}),     we  have     induced  isometries     
     \begin{eqnarray}\label{eq-7798.00a1b}
      {\rm  Conf}(\varphi):     {\rm  Conf}_n(M,r) \longrightarrow     {\rm  Conf}_n(M,r)  
   \end{eqnarray} 
   and 
   \begin{eqnarray}\label{eq-1.00a1b}
      {\rm  Conf}(\varphi)/\Sigma_n:     {\rm  Conf}_n(M,r)/\Sigma_n \longrightarrow     {\rm  Conf}_n(M,r)/\Sigma_n.   
   \end{eqnarray} 
    Consequently,  we  have     
    \begin{eqnarray}\label{eq-ex08a}
       {\rm  Isom}({\rm  Conf}_n(M,r) )&=& {\rm  Isom}(\bigcap_{r\in \mathbb{I}}{\rm  Conf}_n(M,r) )
       \nonumber \\
      &=&
    {\rm  Isom}({\rm  Conf}_n(M,r)/\Sigma_n )     \nonumber\\
    &=& {\rm  Isom}(\bigcap_{r\in \mathbb{I}}{\rm  Conf}_n(M,r)/\Sigma_n )     \nonumber\\
    &=&{\rm  Isom}(M).   
    \end{eqnarray}

    \begin{enumerate}[(1)]
    
    \item
    Let $M=V$  be  a  discrete  set of  cardinality  at  least  $n$.  Then  both     hypergraphs   in  (\ref{eq-ex03a}) 
    are  the  collection  of  all  possible  subsets  of  $V$  of  cardinality  $n$.    
     We  have  
     \begin{eqnarray*}
     {\rm  Stab}(  {\rm  Conf}_n(V,r)/\Sigma_n) &=& {\rm Stab}(\bigcap_{r\in \mathbb{I}}{\rm  Conf}_n(V,r)/\Sigma_n)\\
     &=& 
     \begin{cases}
     \Sigma_n, &  {\rm  ~if~} |V|=n\\
     1,  &{\rm ~if~} |V|>n.  
     \end{cases}
     \end{eqnarray*}
    Consequently,  
    \begin{eqnarray*}
    {\rm  Aut}_{\rm  Isom}(  {\rm  Conf}_n(V,r)/\Sigma_n) &=&  {\rm  Aut}_{\rm  Isom}( \bigcap_{r\in \mathbb{I}} {\rm  Conf}_n(V,r)/\Sigma_n)\\
    &=& 
    \begin{cases}
    1,  & {\rm~if~} |V|=n\\
    {\rm  Isom}(V),  &{\rm ~if~} |V|>n 
    \end{cases}
    \end{eqnarray*}
    where  $  {\rm  Isom}(V)$  is  given  by   (\ref{eq-2.01e}).  
    \item
    Let  $M=S^1$  be  the  unit  circle.   Suppose  $0\leq  r<  \pi/n$.   
    Let $\varphi\in   {\rm  Stab}( {\rm  Conf}_n(S^1,r)/\Sigma_n)$.    
     Then
    $\varphi$   is  a   homeomorphism   $\varphi:  S^1\longrightarrow  S^1$   
    such  that   
    \begin{eqnarray*}
    \varphi(\{x_1,\ldots, x_n\})=  \{x_1,\ldots, x_n\} 
    \end{eqnarray*}
    for  each  
    $
    \{x_1,\ldots, x_n\}\in  {\rm  Conf}_n(S^1,r)/\Sigma_n$.   
    Hence   
    there   exists  $1\leq  i\leq  n$  
    such  that  $\varphi(x_k)=  x_{k+i}$    for  
    each  $1\leq  k\leq  n$  or   $\varphi(x_k)=  x_{-k+i}$    for  
    each  $1\leq  k\leq  n$,   where  $k+n\equiv  k$   mod  $n$.   
    Fix  a   hyperedge  $\{x_1,\ldots, x_n\}\in  {\rm  Conf}_n(S^1,r)/\Sigma_n$.  
      Note  that   
    each   vertex $x_i$  in      $\{x_1,\ldots, x_n\} $  
    can  be  perturbed  continuously  in a small  neighborhood.  
    For  instance,   if  we  perturb  $x_1$  slightly and  take  $x_1'\neq  x_1$  in  a  sufficiently  small  neighborhood  of  $x_1$  in  $S^1$,  
       then  we  obtain  another  hyperedge  
     $
       \{x'_1,x_2, \ldots, x_n\}\in  {\rm  Conf}_n(S^1,r)/\Sigma_n 
   $ 
       such  that  
       \begin{eqnarray*}
       \varphi(x'_1)\notin  \{x'_1, x_2,  \ldots,  x_n\} 
       \end{eqnarray*}
           unless 
       $\varphi$  is  the  identity  map  on  $S^1$.  
     Thus     
    \begin{eqnarray*}
    {\rm  Stab}( {\rm  Conf}_n(S^1,r)/\Sigma_n) =\{1\} 
    \end{eqnarray*}
       is   trivial.  
    Hence   with  the  help of  (\ref{eq-ex08a}),  
   \begin{eqnarray*}
    {\rm   Aut}_{\rm  Isom}( {\rm  Conf}_n(S^1,r)/\Sigma_n)  &=& {\rm  Isom}({\rm  Conf}_n(S^1,r)/\Sigma_n)\\
    &=& {\rm  Isom}(S^1)\\
    &=& O(2)
    \end{eqnarray*}
    is  the  orthogonal  group  of degree  $2$.

    On  the  other  hand,    since  each  vertex   $x_i$  in   a   hyperedge  
    \begin{eqnarray*}
    \{x_1,\ldots, x_n\}\in  \bigcap_{r< \frac{\pi}{n}} {\rm  Conf}_n(S^1,r)/\Sigma_n 
    \end{eqnarray*}
       cannot  be  perturbed  in  any  small  neighborhood,  
    we  have  
    \begin{eqnarray*}
     {\rm  Stab}(\bigcap_{r<  \frac{\pi}{n}} {\rm  Conf}_n(S^1,r)/\Sigma_n) \cap  {\rm  Isom}(S^1)&=&\Sigma_n  \cap    {\rm  Isom}(S^1)\\
     &=&\{  e^{i\frac{2k\pi}{n}},\gamma  e^{i\frac{2k\pi}{n}}\mid  1\leq  k\leq  n\},
    \end{eqnarray*}
    where  $\Sigma_n$  permutes  the  order  of   
    $e^{i\theta}, e^{i(\theta+ \frac{2\pi}{n})}, \ldots, e^{i(\theta+ \frac{2(n-1)\pi}{n})} $  for  any  possible  $\theta$  such  that  
    \begin{eqnarray*}
    \{e^{i\theta}, e^{i(\theta+ \frac{2\pi}{n})}, \ldots, e^{i(\theta+ \frac{2(n-1)\pi}{n})} \}\in  \bigcap_{r<  \frac{\pi}{n}} {\rm  Conf}_n(S^1,r)/\Sigma_n   
    \end{eqnarray*}
    and  $\gamma$  is  the  reflection  with  respect  to  $x$-axis.      
    It  follows with  the  help  of  (\ref{eq-ex08a})  that   
    \begin{eqnarray*}
     {\rm   Aut}_{\rm  Isom}( \bigcap_{r< \frac{\pi}{n}} {\rm  Conf}_n(S^1,r)/\Sigma_n)  
    &=& {\rm  Isom}(S^1)/(\Sigma_n\cap {\rm  Isom}(S^1)) \\
    &=&  O(2)/(\Sigma_n\cap  O(2))\\
    &\cong &  O(2). 
    \end{eqnarray*}
    \item
    Let  $M=S^m$  be  the  unit  $m$-sphere. 
   Choose  $r$  such  that  ${\rm  Conf}_n(S^m,r)/\Sigma_n$  is  non-empty.  
   Then     
    each   vertex $x_i$  in   a   hyperedge  $\{x_1,\ldots, x_n\}\in  {\rm  Conf}_n(S^m,r)/\Sigma_n$  
    can  be  perturbed  continuously  in a  small  neighborhood.   Thus  by  a  similar  argument  
    with     Example~\ref{ex-1-sec5}~(2),    we  have  that 
    ${\rm  Stab}( {\rm  Conf}_n(S^1,r)/\Sigma_n)$   is   trivial.  
    Therefore,  
     \begin{eqnarray*}
    {\rm   Aut}_{\rm  Isom}( {\rm  Conf}_n(S^m,r)/\Sigma_n)  &=& {\rm  Isom}({\rm  Conf}_n(S^m,r)/\Sigma_n)\\
    &=& {\rm  Isom}(S^m)\\
    &=& O(m+1)
    \end{eqnarray*}
    is  the  orthogonal  group  of  degree  $m+1$.  
    \end{enumerate}
     \end{example}

\section{Automorphism  groups  of  configuration  spaces  and  hypergraphs}\label{sec-10}

 In  this  section,   we  study the  automorphism  groups  of  hypergraphs  on  manifolds.  
 We  prove that  the  automorphism  group  of  a hypergraph  is  a  subgroup  in   
   the  intersection  of  the  automorphism  
  groups  of  the  (lower-)associated  simplicial complexes  and  the  (lower-)associated  
  independence  hypergraphs  in  Proposition~\ref{subgrp}.

For  any  directed  $n$-hyperedge  $\vec\sigma(M)= (x_1,\ldots,x_n)$  on  $M$,  
define   its  {\it closure}  $\Delta\vec\sigma(M)$  as   the  hyperdigraph  consisting  of  all  the 
non-empty   subsequences  of  $\vec\sigma(M)$,  i.e.  
\begin{eqnarray*}
\Delta\vec\sigma(M)=\{(x_{i_1},\ldots,x_{i_k})\mid   1\leq  i_1<\cdots<i_k\leq  n{\rm~and~}  1\leq  k\leq n\};    
\end{eqnarray*}
and  define  its   {\it  co-closure}   $\bar\Delta\vec\sigma(M)$  as   
the  hyperdigraph  consisting  of  all  the  
finite    super-sequences   of  $\vec\sigma(M)$,  i.e.  
\begin{eqnarray*}
\bar\Delta\vec\sigma(M)&=&\{(y_1,\ldots,y_l)\mid     l\geq  n  {\rm~and~there~exist~} 
1\leq  i_1<\cdots<i_n\leq  l \\
&&{\rm~such ~that~}  y_{i_1}=x_1,\ldots,  y_{i_n}=x_n\}.     
\end{eqnarray*}

\begin{definition}
\begin{enumerate}[(1)]
\item
 We  call   a  hyperdigraph    $\vec{\mathcal{K}}(M)$   on  $M$  a  {\it  directed  simplicial  complex}  
 if    $\vec{\mathcal{K}}(M)$   is  a  $\Delta$-manifold,  in  other  words,   if   $\Delta\vec\sigma$  is  a  graded  submanifold  of  $\vec{\mathcal{K}}(M)$  
  for  any  $\vec\sigma\in \vec{\mathcal{K}}(M)$;    
  \item 
   We  call   a  hyperdigraph    $\vec{\mathcal{L}}(M)$   on  $M$  a  {\it   independence  hyperdigraph}  
   if   $\bar \Delta\vec\sigma$  is  a  graded  submanifold  of  $\vec{\mathcal{L}}(M)$  
  for  any  $\vec\sigma\in \vec{\mathcal{L}}(M)$.  
  \end{enumerate}
\end{definition}  

 Let  $\vec{\mathcal{H}}(M)$  be  a  hyperdigraph  on  $M$.  
 Define  the   {\it  associated  directed  simplicial  complex}   of  $\vec{\mathcal{H}}(M)$   
 as  
 \begin{eqnarray}\label{eq-2.1aammm}
   \Delta\vec{\mathcal{H}}(M)   =\bigcup_{\vec{\sigma}(M)\in\vec{\mathcal{H}}(M)}\Delta\vec\sigma(M),     
 \end{eqnarray}
 which  is  the  smallest    directed  simplicial  complex  containing  $\vec{\mathcal{H}}(M)$  as  
 a  graded  submanifold.  
   Define  the   {\it   lower-associated  directed  simplicial  complex}   of  $\vec{\mathcal{H}}(M)$   
 as  
 \begin{eqnarray}\label{eq-2.1bbmmm}
   \delta\vec{\mathcal{H}}(M)   =\bigcup_{\Delta\vec{\sigma}(M)\subseteq\vec{\mathcal{H}}(M)} \vec\sigma(M),     
 \end{eqnarray}
 which  is  the    largest    directed  simplicial  complex  contained  in  $\vec{\mathcal{H}}(M)$  as  
 a  graded  submanifold.  
  Define  the   {\it  associated   independence  hyperdigraph}    of  $\vec{\mathcal{H}}(M)$   
 as  
 \begin{eqnarray}\label{eq-2.1ccmmm}
  \bar \Delta\vec{\mathcal{H}}(M)   =\bigcup_{\vec{\sigma}(M)\in\vec{\mathcal{H}}(M)}\bar\Delta\vec\sigma(M),     
 \end{eqnarray}
 which  is  the  smallest  independence  hyperdigraph   containing  $\vec{\mathcal{H}}(M)$  as  
 a  graded  submanifold.  
   Define  the   {\it   lower-associated  independence  hyperdigraph}   of  $\vec{\mathcal{H}}(M)$   
 as  
 \begin{eqnarray}\label{eq-2.1ddmmm}
   \bar\delta\vec{\mathcal{H}}(M)   =\bigcup_{\bar\Delta\vec{\sigma}(M)\subseteq\vec{\mathcal{H}}(M)} \vec\sigma(M),     
 \end{eqnarray}
 which  is  the    largest  independence  hyperdigraph  contained  in  $\vec{\mathcal{H}}(M)$  as  
 a  graded  submanifold.  
 A  directed  hyperedge  $\vec\sigma(M)\in \vec{\mathcal{H}}(M)$  is  called  {\it  maximal}  if  
  there  does  not  exist  any  $\vec\tau(M)  \in \vec{\mathcal{H}}(M)$  
  such  that    $\vec\tau(M) \neq \vec\sigma(M)$  and  
       $\vec\sigma(M)$  is  a  subsequence  of   $\vec\tau(M) $.    
 Let  $\max(\vec{\mathcal{H}}(M))$  be  the  collection  of  all  the  maximal   directed 
  hyperedges  in  $\vec{\mathcal{H}}(M)$. 
  If   $\vec{\mathcal{H}}(M)= \bigcup_{k=1}^N\vec{\mathcal{H}}_k(M)$ where  $N$  is  finite,  
  then  for 
  each  directed hyperedge  $\vec{\sigma}(M)\in \vec{\mathcal{H}}(M)$  
  there  exists  a  maximal  directed  hyperedge  $\vec\eta(M)\in \max \vec{\mathcal{H}}(M)$  
  such  that  $\vec\sigma(M)$  is  a  subsequence  of  $\vec \eta(M)$,   and  consequently   
 \begin{eqnarray*}
 \max\vec{\mathcal{H}}(M)= \max(\Delta\vec{\mathcal{H}}(M)),~~~~~~
 \Delta\vec{\mathcal{H}}(M)=\Delta(\max\vec{\mathcal{H}}(M)). 
 \end{eqnarray*}
   A   directed  hyperedge  $\vec\sigma\in \vec{\mathcal{H}}(M)$  is  called  {\it  minimal}   if  
  there  does  not  exist  any  $\vec{\tau}   \in\vec{\mathcal{H}}(M)$  such that  
    $\vec\tau(M) \neq \vec\sigma(M)$  and  
       $\vec\tau(M)$  is  a  subsequence  of   $\vec\sigma(M) $.      
 Let  $\min(\vec{\mathcal{H}}(M))$  be  the  collection  of  all  the  minimal  directed  
 hyperedges  in  $\vec{\mathcal{H}}(M)$. 
  Then  
  \begin{eqnarray*}
  \min \vec{\mathcal{H}}(M)=\min(\bar\Delta\vec{\mathcal{H}}(M)),~~~~~~
  \bar\Delta\vec{\mathcal{H}}(M)=\bar\Delta(\min\vec{\mathcal{H}}(M)).  
  \end{eqnarray*}

For  any     $n$-hyperedge  $ \sigma(M)= \{x_1,\ldots,x_n\}$  on  $M$,  
define   its  {\it closure}  $\Delta \sigma(M)$  as   the  hypergraph  consisting  of  all  the  non-empty  subsets  of  
$ \sigma(M)$,  i.e.  
\begin{eqnarray*}
\Delta \sigma(M)=\{\{x_{i_1},\ldots,x_{i_k}\}\mid   1\leq  i_1<\cdots<i_k\leq  n{\rm~and~}  1\leq  k\leq n\};    
\end{eqnarray*}
and  define  its   {\it  co-closure}   $\bar\Delta \sigma(M)$  as   
the  hypergraph  consisting  of  all  the  
finite    super-sets   of  $\sigma(M)$,  i.e.  
\begin{eqnarray*}
\bar\Delta \sigma(M)&=&\{\{y_1,\ldots,y_l\}\mid     l\geq  n  {\rm~and~there~exist~} 
1\leq  i_1<\cdots<i_n\leq  l \\
&&{\rm~such ~that~}  y_{i_1}=x_1,\ldots,  y_{i_n}=x_n\}.     
\end{eqnarray*}

\begin{definition}
\begin{enumerate}[(1)]
\item
 We  call   a  hypergraph    ${\mathcal{K}}(M)$   on  $M$  a  {\it   simplicial  complex}  
 if    ${\mathcal{K}}(M)$   is  a  $\Delta$-manifold,  in  other  words,   if   $\Delta\sigma$  is  a  graded  submanifold  of  ${\mathcal{K}}(M)$  
  for  any  $\sigma\in {\mathcal{K}}(M)$;    
  \item 
   We  call   a  hypergraph    ${\mathcal{L}}(M)$   on  $M$  an  {\it   independence  hypergraph}  
   if   $\bar \Delta\sigma$  is  a  graded  submanifold  of  ${\mathcal{L}}(M)$  
  for  any  $\sigma\in{\mathcal{L}}(M)$.  
  \end{enumerate}
\end{definition}

 Let  $\mathcal{H}(M)=\pi_\bullet(\vec{\mathcal{H}}(M))$.  
 Then  $\mathcal{H}(M)$    is   a  hypergraph  on  $M$.  
 Define  the   {\it  associated  simplicial  complex}   of  $ {\mathcal{H}}(M)$   
 as  
 \begin{eqnarray}\label{eq-7.1aammm}
   \Delta {\mathcal{H}}(M)   =\bigcup_{ {\sigma}(M)\in {\mathcal{H}}(M)}\Delta \sigma(M),     
 \end{eqnarray}
 which  is  the  smallest      simplicial  complex  containing  $ {\mathcal{H}}(M)$  as  
 a  graded  submanifold.  
   Define  the   {\it   lower-associated    simplicial  complex}   of  $ {\mathcal{H}}(M)$   
 as  
 \begin{eqnarray}\label{eq-7.1bbmmm}
   \delta {\mathcal{H}}(M)   =\bigcup_{\Delta {\sigma}(M)\subseteq {\mathcal{H}}(M)}  \sigma(M),     
 \end{eqnarray}
 which  is  the    largest       simplicial  complex  contained  in  $ {\mathcal{H}}(M)$  as  
 a  graded  submanifold.  
  Define  the   {\it  associated   independence  hypergraph}    of  ${\mathcal{H}}(M)$   
 as  
 \begin{eqnarray}\label{eq-7.1ccmmm}
  \bar \Delta {\mathcal{H}}(M)   =\bigcup_{ {\sigma}(M)\in {\mathcal{H}}(M)}\bar\Delta \sigma(M),     
 \end{eqnarray}
 which  is  the  smallest  independence  hypergraph   containing  $ {\mathcal{H}}(M)$  as  
 a  graded  submanifold.  
   Define  the   {\it   lower-associated  independence  hypergraph}   of  $ {\mathcal{H}}(M)$   
 as  
 \begin{eqnarray}\label{eq-7.1ddmmm}
   \bar\delta {\mathcal{H}}(M)   =\bigcup_{\bar\delta {\sigma}(M)\subseteq {\mathcal{H}}(M)}  
   \sigma(M),     
 \end{eqnarray}
 which  is  the    largest  independence  hypergraph  contained  in  $ {\mathcal{H}}(M)$  as  
 a  graded  submanifold.  
  A    hyperedge  $ \sigma(M)\in {\mathcal{H}}(M)$  is  called  {\it  maximal}  if  
  there  does  not  exist  any  $ \tau(M)  \in  {\mathcal{H}}(M)$  
  such  that      
       $ \sigma(M)$  is  a  proper  subset  of   $ \tau(M) $.    
 Let  $\max {\mathcal{H}}(M)$  be  the  collection  of  all  the  maximal    
  hyperedges  in  $ {\mathcal{H}}(M)$. 
    If   ${\mathcal{H}}(M)= \bigcup_{k=1}^N{\mathcal{H}}_k(M)$ where  $N$  is  finite,  then  
 \begin{eqnarray*}
 \max {\mathcal{H}}(M)= \max(\Delta {\mathcal{H}}(M)),~~~~~~
 \Delta {\mathcal{H}}(M)=\Delta(\max {\mathcal{H}}(M)). 
 \end{eqnarray*}
   A     hyperedge  $ \sigma\in  {\mathcal{H}}(M)$  is  called  {\it  minimal}   if  
  there  does  not  exist  any  $ {\tau}   \in {\mathcal{H}}(M)$  such that  
       $ \tau(M)$  is  a   proper  subset   of   $ \sigma(M) $.      
 Let  $\min {\mathcal{H}}(M)$  be  the  collection  of  all  the  minimal  directed  
 hyperedges  in  $ {\mathcal{H}}(M)$. 
  Then  
  \begin{eqnarray*}
  \min   {\mathcal{H}}(M) =\min(\bar\Delta {\mathcal{H}}(M)),~~~~~~
  \bar\Delta {\mathcal{H}}(M)=\bar\Delta(\min {\mathcal{H}}(M)).  
  \end{eqnarray*}
  By  (\ref{eq-2.1aammm})   -  (\ref{eq-7.1ddmmm}),   it  is  straight-forward  that  
 \begin{eqnarray*}
 &\Delta\mathcal{H}(M) = \pi_\bullet(\Delta\vec{\mathcal{H}}(M)),   ~~~&
\delta\mathcal{H}(M) = \pi_\bullet(\delta\vec{\mathcal{H}}(M)),\\ 
&\bar\Delta\mathcal{H}(M) = \pi_\bullet(\bar\Delta\vec{\mathcal{H}}(M)),  ~~~& 
\bar\delta\mathcal{H}(M) = \pi_\bullet(\bar\delta\vec{\mathcal{H}}(M))
\end{eqnarray*}
and  
\begin{eqnarray*}
   \max \mathcal{H}(M) & =&\pi_\bullet ( \max  \vec{\mathcal{H}}(M)), \\      
  \min \mathcal{H}(M) &  =&\pi_\bullet  (\min \vec{\mathcal{H}}(M)).  
\end{eqnarray*}
 
\begin{proposition}\label{subgrp}
 Suppose  ${\mathcal{H}}(M)= \bigcup_{k=1}^N{\mathcal{H}}_k(M)$ where  $N$  is  finite.  Then  
 \begin{enumerate}[(1)]
 \item
 ${\rm  Aut}(\max\mathcal{H}(M))= {\rm  Aut}(\Delta\mathcal{H}(M))$;
 \item
 ${\rm  Aut}(\min\mathcal{H}(M))= {\rm  Aut}(\bar\Delta\mathcal{H}(M))$; 
 \item
 ${\rm  Aut}(\mathcal{H}(M))$  is  a  subgroup  of    
 \begin{eqnarray*}
 {\rm  Aut}(\Delta\mathcal{H}(M)) \cap  {\rm  Aut}(\bar\Delta\mathcal{H}(M))  \cap  {\rm  Aut}(\delta\mathcal{H}(M))\cap  {\rm  Aut}(\bar\delta\mathcal{H}(M)).  
 \end{eqnarray*}
  \end{enumerate}  
\end{proposition}

\begin{proof}
 (1)  Let  $\varphi\in  {\rm  Aut}(\Delta\mathcal{H}(M))$.  
 Then  for  any  $\sigma(M)\in \max\mathcal{H}(M)$,   
 \begin{eqnarray*}
 \varphi(\sigma(M))\in \max\mathcal{H}(M). 
 \end{eqnarray*}   
 Thus  $\varphi\in  {\rm  Aut}(\max\mathcal{H}(M))$.  
 Hence  
${\rm  Aut}(\Delta\mathcal{H}(M))\subseteq  {\rm  Aut}(\max\mathcal{H}(M))$.

 Conversely,  let   $\psi\in  {\rm  Aut}(\max\mathcal{H}(M))$.    
 For  any  $\tau(M)\in \Delta\mathcal{H}(M)$,  there exists  $\sigma(M)\in  \max\mathcal{H}(M)$  such  that  
 $\tau(M)\subseteq \sigma(M)$.  
 Hence    
 \begin{eqnarray}\label{eq-99.7a1}
 \psi(\tau(M)) \subseteq\psi(\sigma(M)).  
 \end{eqnarray}
Note  that
 \begin{eqnarray}\label{eq-99.7a2}
 \Delta\psi(\sigma(M))\subseteq \Delta\mathcal{H}(M).    
 \end{eqnarray}
    It  follows  from  (\ref{eq-99.7a1})  and     (\ref{eq-99.7a2})  that 
     \begin{eqnarray*}
     \psi(\tau(M))\in  \Delta\mathcal{H}(M).  
     \end{eqnarray*}
      Thus  $\psi\in   {\rm  Aut}(\Delta\mathcal{H}(M))$.  
 Hence     ${\rm  Aut}(\max\mathcal{H}(M))\subseteq {\rm  Aut}(\Delta\mathcal{H}(M))$.

 (2)   Let  $\varphi\in  {\rm  Aut}(\bar\Delta\mathcal{H}(M))$.  
 Then   for  any  $\sigma(M)\in \min\mathcal{H}(M)$,   
 \begin{eqnarray*}
 \varphi(\sigma(M))\in \min\mathcal{H}(M).
 \end{eqnarray*}    
 Thus  $\varphi\in  {\rm  Aut}(\min\mathcal{H}(M))$.  
 Hence  ${\rm  Aut}(\bar\Delta\mathcal{H}(M))\subseteq {\rm  Aut}(\min\mathcal{H}(M))$.

 Conversely,  let   $\psi\in  {\rm  Aut}(\min\mathcal{H}(M))$.    
 For  any  $\tau(M)\in \bar\Delta\mathcal{H}(M)$,  
 there exists  $\sigma(M)\in  \min\mathcal{H}(M)$  such  that  
 $\tau(M)\supseteq \sigma(M)$.  
Hence 
\begin{eqnarray}\label{eq-99.7b1}
\psi(\tau(M)) \supseteq\psi(\sigma(M)). 
\end{eqnarray}
Note  that 
\begin{eqnarray}\label{eq-99.7b2}
\bar\Delta\psi(\sigma(M))\subseteq\bar \Delta\mathcal{H}(M). 
\end{eqnarray}
   It  follows  from  (\ref{eq-99.7b1})  and     (\ref{eq-99.7b2})  that 
\begin{eqnarray*}
 \psi(\tau(M))\in\bar\Delta\mathcal{H}(M).  
\end{eqnarray*}
  Thus  $\psi\in   {\rm  Aut}(\bar\Delta\mathcal{H}(M))$.  
Hence    ${\rm  Aut}(\min\mathcal{H}(M))\subseteq  {\rm  Aut}(\bar\Delta\mathcal{H}(M))$.

 (3)  Let  $\varphi\in   {\rm  Aut}(\mathcal{H}(M))$.   
 With  the  help  of   (1)  and   (2),  in  order to  prove   (3),  
 it  suffices to  prove   $\varphi \in {\rm  Aut}(\max\mathcal{H}(M))$,    $\varphi \in {\rm  Aut}(\min\mathcal{H}(M))$,    $\varphi \in {\rm  Aut}(\delta\mathcal{H}(M))$  and   $\varphi \in {\rm  Aut}(\bar\delta\mathcal{H}(M))$.  
 
 \begin{enumerate}[(i)]
 \item
 Let  $\sigma_1(M)\in  \max\mathcal{H}(M)$.   Then  $\varphi(\sigma_1(M))\in    \max\mathcal{H}(M)$.  
 Thus  $\varphi\in  {\rm  Aut}(\max\mathcal{H}(M))$.  
 
 \item
  Let  $\sigma_2(M)\in  \min\mathcal{H}(M)$.   Then  $\varphi(\sigma_2(M))\in    \min\mathcal{H}(M)$.  
 Thus  $\varphi\in  {\rm  Aut}(\min\mathcal{H}(M))$.  
 
 \item
 Let  $\sigma_3(M)\in  \delta\mathcal{H}(M)$.   Then  
 $\tau(M)\in  \mathcal{H}(M)$  for  any  $\tau(M)\subseteq\sigma_3(M)$  with  $\tau(M)\neq\emptyset$.  
 Since $\varphi\in   {\rm  Aut}(\mathcal{H}(M))$,  it  follows  that 
    $\varphi(\tau(M))\in  \mathcal{H}(M) $   for any  $ \tau(M) \subseteq \sigma_3(M) $  with  $\tau(M)\neq\emptyset$.   
    Let  $\tau'(M)= \varphi(\tau(M))$.  
    Then  $ \tau'(M)\in  \mathcal{H}(M) $   for any  $ \tau'(M)\subseteq\varphi(\sigma_3(M))$  with  $\tau'(M)\neq\emptyset$. 
  Consequently,   $\varphi(\sigma_3(M)) \in \delta\mathcal{H}(M)$.  
 Thus     $\varphi \in {\rm  Aut}(\delta\mathcal{H}(M))$.  
 
 \item
  Let  $\sigma_4(M)\in  \bar\delta\mathcal{H}(M)$.   Then  
 $\tau(M)\in  \mathcal{H}(M)$  for  any  $\tau(M)\supseteq\sigma_4(M)$  with  $\tau(M)\in  {\rm  Conf}_\bullet(M)/\Sigma_\bullet$.  
 Since $\varphi\in   {\rm  Aut}(\mathcal{H}(M))$,  it  follows  that    $\varphi(\tau(M))\in  \mathcal{H}(M) $   for any  $\tau(M)\supseteq\varphi(\sigma_4(M))$  with   $\tau(M)\in  {\rm  Conf}_\bullet(M)/\Sigma_\bullet$.   Let  $\tau'(M)= \varphi(\tau(M))$.  
 Then  
 $ \tau'(M)\in  \mathcal{H}(M) $   for any  $ \tau'(M)\supseteq\varphi(\sigma_4(M))$ 
  with  $\tau'(M)\in  {\rm  Conf}_\bullet(M)/\Sigma_\bullet$.  
 Consequently,   $\varphi(\sigma_4) \in \bar\delta\mathcal{H}(M)$.  
 Thus     $\varphi \in {\rm  Aut}(\bar\delta\mathcal{H}(M))$.  

\end{enumerate}
Summarizing  (i)  - (iv),  we  obtain  (3).  
\end{proof}

\begin{remark}
Let  $M=V$  be  a  discrete  set.  
The  associated simplicial complex  is  studied  in   \cite{parks,h1}; 
 the    lower-associated simplicial complex    is  studied  in    \cite{jktr2,jktr2023}; 
  the   associated  independence hypergraph  as  well as  the  lower-associated  independence hypergraph   is  studied  in    \cite{jktr2023}.  
\end{remark}

\section{Covering  maps  of  hypergraphs  and  order  of  the  associated  vector bundles}\label{bundle}

In  this  section,  
we   estimate  the  order  of  the  associated  vector  bundles 
of  the  covering  maps  from  $\Sigma_\bullet$-invariant  hyperdigraphs  to  the  
underlying  hypergraphs,  in  Theorem~\ref{pr-order-a}.

Let  $\Sigma_n$  act   on  $\mathbb{R}^n$  from  right  by  permuting  the  coordinates.   
Then  by     (\ref{eq-0.a1}),  
we  have  a   vector  bundle 
\begin{eqnarray*}
\xi ^n(M):  \mathbb{R}^n\longrightarrow   {\rm  Conf}_n(M)\times_{\Sigma_n}\mathbb{R}^n
\longrightarrow   {\rm  Conf}_n(M)/\Sigma_n. 
\end{eqnarray*}

   \begin{definition}
   We  call  a  hyperdigraph  $\vec{\mathcal{H}}(M)$  on  $M$  {\it  $\Sigma_\bullet$-invariant}  
   if  for  each  $n\in  \mathbb{N}$,  we  have  
  $ s(\vec\sigma(M))\in  \vec{\mathcal{H}}_n(M)$
   for  any   $s\in \Sigma_n$  and  any  $\vec\sigma(M)\in    \vec{\mathcal{H}}_n(M)$.  
   Here  $s$  acts  on   $\vec\sigma(M)$  by  permuting  the  order  of  the  coordinates  from  left.  
   \end{definition}
   
    Let  $\vec{\mathcal{H}}(M)$  be a  $\Sigma_\bullet$-invariant   hyperdigraph  on  $M$.
Let   $\mathcal{H} (M)=\pi_\bullet(  \vec{\mathcal{H}} (M))$.   Then  we  have  a   vector  bundle 
\begin{eqnarray}\label{eq-0.a2}
\xi (\mathcal{H}_n(M)):  \mathbb{R} ^n\longrightarrow \vec{\mathcal{H}}_n(M)\times_{\Sigma_n}\mathbb{R}^n
\longrightarrow   \mathcal{H}_n(M).    
\end{eqnarray}

  \begin{lemma}\label{le-99.7ai}
  \begin{enumerate}[(1)]
    \item
  $\vec{\mathcal{H}}_n(M)=  \pi_n^{-1}(\mathcal{H}_n(M))$      and  
    \begin{eqnarray}\label{eq-moxqfi5}
    \pi_n:  \vec{\mathcal{H}}_n(M)\longrightarrow  \mathcal{H}_n(M)
    \end{eqnarray}  
  is  an  $n!$-sheeted  covering  map;  
  \item
  the  vector  bundle  
  (\ref{eq-0.a2})   
is   a  pull-back  bundle  
  \begin{eqnarray}\label{eq-0.a5}
  \iota^*(\xi ^n(M))  = \xi (\mathcal{H}_n(M))  
  \end{eqnarray}
    where    $\iota:  \mathcal{H}_n(M)\longrightarrow   {\rm  Conf}_n(M)/\Sigma_n$  
   is  the  canonical  embedding.  
  \end{enumerate}
  \end{lemma}
    
    \begin{proof}
    (1)  
    Let  $\vec\sigma(M)\in  \vec{\mathcal{H}}_n(M)$.  
    Suppose  $\vec\sigma(M) = (x_1,\ldots,x_n)$  where  $x_1,\ldots,x_n$  are  distinct  
    points in  $M$.    Let  
    $\sigma(M)=\pi_n(\vec\sigma(M))$.   Then  $\sigma(M)\in  \mathcal{H}_n(M)$
    and   $\sigma(M) =   \{x_1,\ldots,x_n\}$.  
    Thus  
    \begin{eqnarray*}
    \pi_n^{-1}(\sigma(M) )=  \{(x_{s(1)},\ldots,x_{s(n)})\mid    s\in  \Sigma_n\}.  
    \end{eqnarray*}
    Since  $\vec{\mathcal{H}}(M)$  be a  $\Sigma_\bullet$-invariant,  
    we  have    $(x_{s(1)},\ldots,x_{s(n)})\in  \vec{\mathcal{H}}(M)$ 
    for  any  $s\in  \Sigma_n$.  The  $n!$-sheeted  covering  map  
    \begin{eqnarray*}
    \pi_n:   \pi_n^{-1}(\sigma(M) )\longrightarrow  \{\sigma(M)\} 
    \end{eqnarray*}
    is  the  restriction  of  (\ref{eq-moxqfi5})  to  $  \pi_n^{-1}(\sigma(M) )$.  
    Thus  (\ref{eq-moxqfi5})  is  an  $n!$-sheeted  covering  map.

    (2) Since $\vec{\mathcal{H}}(M)$  is   $\Sigma_\bullet$-invariant,   
   the  canonical  embedding
   $ \tilde\iota:  \vec{ \mathcal{H}}_n(M)\longrightarrow  {\rm  Conf}_n(M)$    and  
    the  canonical  embedding  $\iota:  \mathcal{H}_n(M)\longrightarrow   {\rm  Conf}_n(M)/\Sigma_n$   satisfy    a  commutative  diagram  
    \begin{eqnarray*}
    \xymatrix{
       \vec{ \mathcal{H}}_n(M)\ar[r]^-{\tilde \iota} \ar[d]_-{\pi_n} & {\rm  Conf}_n(M) \ar[d]^-{\pi_n}
\\
    \mathcal{H}_n(M)\ar[r]^-{\iota}  & {\rm  Conf}_n(M)/\Sigma_n.
    }
    \end{eqnarray*}
    This  diagram  induces a  pull-back  of  vector  bundles 
    (\ref{eq-0.a5}).  
    \end{proof}

  \begin{lemma}\label{le-0a1a}
  For  any  vector  bundle  $\xi$,  let   the  order   $o(\xi)$    be  the  smallest  positive  integer     
  such  that   $\xi^{\oplus  o(\xi)}$,   the  $o(\xi)$-fold   self  Whitney  sum  of  $\xi$,
      is  a  trivial  bundle.  
  Then  
  \begin{eqnarray*}
  o(\xi (\mathcal{H}_n(M))  ) \mid   o(\xi  ^n(M)).     
  \end{eqnarray*}
   \end{lemma}
  
  \begin{proof}
  For  any  positive  integer  $k$,  if  $\xi  ^n(M)^{\oplus  k}$  is  a    trivial  bundle, 
   then  with  the  help  of  (\ref{eq-0.a5}),   we  have  that 
  \begin{eqnarray*}
 \xi  (\mathcal{H}_n(M)) ^{\oplus  k}= \iota^*(\xi ^n(M)) ^{\oplus  k}=  \iota^*(\xi ^n(M)^{\oplus  k}) 
  \end{eqnarray*}
  is  a     trivial  bundle.  Let  $k=o(\xi  ^n(M))$.  
  It  follows  that    $o(\xi  (\mathcal{H}_n(M))  )$  is  a  factor  of  $o(\xi  ^n(M))$.  
  We  obtain  the  lemma. 
  \end{proof}

 \begin{theorem} \label{pr-order-a}
  For any closed orientable Riemann surface $S$ whose genus is greater than or equal
to one, 
 \begin{eqnarray}\label{eq-pr-laza2}
  o(\xi(\mathcal{H}_n(S))  )\mid   4.   
  \end{eqnarray}
  Moreover,  
  let  $a_{m,n}=2^{\rho(m-1)}\prod_{3\leq  p\leq n    {\rm~is~prime}} p^{[\frac{m-1}{2}]}$   where   $\rho(m-1)$ is  the number of positive integers less than or equal to $m-1$ that are congruent to $0, 1, 2$ or $4$ mod $8$.  Then 
    \begin{eqnarray}\label{eq-pr-laza1}
   &   o(\xi(\mathcal{H}_n(\mathbb{R}^m))  )\mid  a_{m,n},    
  \\
  \label{eq-pr-laza3}
  &  o(\xi(\mathcal{H}_n(  S^m))  )\mid    2^{\rho(m)-\rho(m-1)} a_{m,n}. 
    \end{eqnarray} 
   Furthermore, 
    Let  $N(m)$   be   the smallest integer $N$ such that $\mathbb{R}P^m$  can be embedded into $\mathbb{R}^N$.  Then 
       \begin{eqnarray}\label{eq-pr-laza4}
   &  o(\xi(\mathcal{H}_n(  \mathbb{R}P^m))  )\mid  2^{  \rho(N (m)-1)-\rho(m)} a_{m+1,n},  \\
    &  o(\xi(\mathcal{H}_n(  \mathbb{R}P^m\times \mathbb{R}^k))  )\mid  2^{ \rho(N (m)+k-1)} \prod_{3\leq  p\leq n    {\rm~is~prime}}  p^{ \frac{m+k-1}{2}}. 
    \label{eq-pr-laza5}
       \end{eqnarray} 
 \end{theorem}
 
 \begin{proof}
   Let   $S$  be     any closed orientable Riemann surface  whose genus is greater than or equal
to one.   By  \cite[Theorem~1.1]{bundle1989},    we  have   
\begin{eqnarray*}
o(\xi^n(S))= 4.
\end{eqnarray*}  
 Thus  by  Lemma~\ref{le-0a1a},   we  obtain  (\ref{eq-pr-laza2}).  
  By  \cite[Theorem~1.1]{bundle1983}, \cite[Theorem~1.2]{bundle1978},   \cite[Theorem~1.1, Theorem~1.2]{swy}  and  \cite[Theorem~5.1]{forum},  
  we  have  
  \begin{eqnarray*}
  o(\xi^n(\mathbb{R}^m))= a_{m,n}. 
  \end{eqnarray*} 
  Thus  by  Lemma~\ref{le-0a1a},     we  obtain  (\ref{eq-pr-laza1}).   
   By   \cite[Theorem~1.1~(b)]{forum},  we  have   
   \begin{eqnarray*}
   o(\xi^n(S^m))= 2^{\rho(m)-\rho(m-1)} a_{m,n}. 
   \end{eqnarray*}
     Thus by  Lemma~\ref{le-0a1a},     we  obtain  (\ref{eq-pr-laza3}).    
  By  \cite[Theorem~1.1]{osaka},  we  have  
  \begin{eqnarray*}
  o(\xi^n(\mathbb{R}P^m))\mid  2^{  \rho(N (m)-1)-\rho(m)} a_{m+1,n}.  
  \end{eqnarray*}
   Thus  by  Lemma~\ref{le-0a1a},    we  obtain   (\ref{eq-pr-laza4}).  
      By  \cite[Theorem~1.4]{osaka},  we  have  
  \begin{eqnarray*}
  o(\xi^n(\mathbb{R}P^m\times  \mathbb{R}^k))\mid  2^{ \rho(N (m)+k-1)} \prod_{3\leq  p\leq n    {\rm~is~prime}}  p^{ \frac{m+k-1}{2}}.  
  \end{eqnarray*} 
 Thus  by  Lemma~\ref{le-0a1a},       we  obtain   (\ref{eq-pr-laza5}).  
 \end{proof}

  \section{Hypergraph  obstructions  for  regular  embeddings}\label{sec-regular}

  In  this  section,  we  generalize  the 
  topological  obstructions  for  regular  embeddings  in  \cite{high1,high2,cohen1}   
  and  give  obstructions  for  regular  embeddings  by  using    hypergraphs  on  manifolds.

Let  $M$  be  a    manifold.  
Let  $\mathbb{F}= \mathbb{R}$  or  $\mathbb{C}$.  
A   real  or  complex   {\it  $k$-regular  embedding}  $f:  M\longrightarrow  \mathbb{F}^N$
 is  an  embedding  of  $M$  into  $\mathbb{F}^N$  such  that  
 for  any  distinct  $k$-points  in  $M$,  
 their  images  are  linearly  independent  in  $\mathbb{F}^N$  (cf.  \cite{high1,high2,cohen1,
 handel1,handel2,handel3}).  
 Let  $\iota:  G_k(\mathbb{F}^N)\longrightarrow  G_k(\mathbb{F}^\infty)$  
 be  the  canonical  embedding  of  Grassmannians  induced  by  the  inclusion  of  
 $\mathbb{F}^N$  in  $\mathbb{F}^\infty$.  
  A   $k$-regular  embedding   $f:  M\longrightarrow  \mathbb{F}^N$
  will  induce   a  commutative  diagram  
 \begin{eqnarray}\label{eq-int-7}
 \xymatrix{
{ \rm  Conf}_k(M)/\Sigma_k \ar[r]^-{f} \ar[rd]_-{\iota\circ  f}  &G_k(\mathbb{F}^N)
 \ar[d]^-{\iota}\\
 & G_k(\mathbb{F}^\infty)
 }
 \end{eqnarray}
 such  that  $\iota\circ  f$  is  the  classifying  map  of  the  vector  bundle  
 \begin{eqnarray*}
 \xi^k(M)\otimes\mathbb{F}:  \mathbb{F}^k\longrightarrow  { \rm  Conf}_k(M)\times_{\Sigma_k} \mathbb{F}^k\longrightarrow   { \rm  Conf}_k(M)/\Sigma_k.  
 \end{eqnarray*}
 Thus  when  $\mathbb{F}=\mathbb{R}$,  the  Stiefel-Whitney  class   of    $\xi^k(M)$  is  given  by  the  pull-back  
 \begin{eqnarray}\label{eq-intro-1}
 (\iota\circ  f)^*:   H^*(G_k(\mathbb{R}^\infty)) \longrightarrow   H^*({ \rm  Conf}_k(M)/\Sigma_k ), 
 \end{eqnarray}
 and   when  $\mathbb{F}=\mathbb{C}$, 
  the  Chern  class   of  $\xi^k(M)\otimes \mathbb{C}$  is  given  by  the  pull-back  
 \begin{eqnarray}\label{eq-intro-2}
 (\iota\circ  f)^*:   H^*(G_k(\mathbb{C}^\infty)) \longrightarrow   H^*({ \rm  Conf}_k(M)/\Sigma_k ).  
 \end{eqnarray}
 Some  obstructions  for  $k$-regular  embeddings  are  given  
 by  (\ref{eq-intro-1})  and   (\ref{eq-intro-2}),  in  the  next  lemma.  
 
 \begin{lemma}[cf. \cite{high1,high2,cohen1}]
 \label{th-intro-1}
 \begin{enumerate}[(1)]
 \item
 Let  $f: M\longrightarrow \mathbb{R}^N$  be  a $k$-regular embedding. If  the  $t$-th   dual  Stiefel-Whitney  class
  \begin{eqnarray}\label{eq-intro-97}
  \bar w_t(\xi^k(M))\neq 0, 
  \end{eqnarray}
  then $N\geq t+k$; 
  \item
  Let  $f: M\longrightarrow \mathbb{C}^N$  be  a  complex $k$-regular embedding. If  the  $t$-th   dual   Chern  class
 \begin{eqnarray}\label{eq-intro-98} 
 \bar  c_t(\xi^k(M))\neq 0,
 \end{eqnarray}
  then $N\geq t+k$.    
  \end{enumerate}
 \end{lemma}
 
 If  we  consider  hyper(di)graphs  on  $M$,  then  more  obstructions  for  $k$-regular  embeddings  
 can  be  obtained.  
 Let  $\vec{\mathcal{H}}_k(M)$  be  any  $\Sigma_k$-invariant 
   $k$-uniform  hyperdigraph  on  $M$.  
   Suppose   $\vec{\mathcal{H}}_k(M)$  is  a  submanifold  of  ${\rm  Conf}_k(M)$.  
 Let  $\mathcal{H}_k(M)= \vec{\mathcal{H}}_k(M)/\Sigma_k$.   
 Then  $ \mathcal{H}_k(M)$  is  a  submanifold  of  ${\rm  Conf}_k(M)/\Sigma_k$. 
 Moreover,  (\ref{eq-int-7})  induces  a  commutative  diagram  
  \begin{eqnarray}\label{eq-int-71}
 \xymatrix{
\mathcal{H}_k(M) \ar[r]^-{f} \ar[rd]_-{\iota\circ  f}  &G_k(\mathbb{F}^N)
 \ar[d]^-{\iota}\\
 & G_k(\mathbb{F}^\infty)
 }
 \end{eqnarray}
 such  that  $\iota\circ  f$  is  the  classifying  map  of  the  vector  bundle  
 \begin{eqnarray*}
 \xi (\mathcal{H}_k(M))\otimes \mathbb{F}:  \mathbb{F}^k\longrightarrow  \vec{\mathcal{H}}_k(M)\times_{\Sigma_k} \mathbb{F}^k\longrightarrow   \mathcal{H}_k(M).  
 \end{eqnarray*}
 Thus  when  $\mathbb{F}=\mathbb{R}$,  the  Stiefel-Whitney  class  of    $ \xi (\mathcal{H}_k(M))$  is  given  by  the  pull-back  
 \begin{eqnarray}\label{eq-intro-8}
 (\iota\circ  f)^*:   H^*(G_k(\mathbb{R}^\infty)) \longrightarrow   H^*( \mathcal{H}_k(M) ),  
 \end{eqnarray}
 and   when  $\mathbb{F}=\mathbb{C}$, 
  the  Chern  class   of  $ \xi (\mathcal{H}_k(M))\otimes \mathbb{C} $  is  given  by  the  pull-back  
 \begin{eqnarray}\label{eq-intro-9}
 (\iota\circ  f)^*:   H^*(G_k(\mathbb{C}^\infty)) \longrightarrow   H^*( \mathcal{H}_k(M)).  
 \end{eqnarray}
Here  in  (\ref{eq-intro-8})  and  (\ref{eq-intro-9}),     $H^*( \mathcal{H}(M))$  
is  the  usual  cohomology  ring  of  manifolds.  
Lemma~\ref{th-intro-1}  can  be  generalized     in    the  next  proposition.  
 
 \begin{proposition} 
 \label{th-intro-2}
 \begin{enumerate}[(1)]
 \item
 Let  $f: M\longrightarrow \mathbb{R}^N$  be  a $k$-regular embedding. If  
 there  exists  a  $k$-uniform  hypergraph  $\mathcal{H}_k(M)$  on  $M$  such  that  
 $\vec{\mathcal{H}}_k(M)=\pi^{-1}(\mathcal{H}_k(M))$  and  
 the  $t$-th   dual  Stiefel-Whitney  class
  \begin{eqnarray}\label{eq-intro-91}
  \bar w_t(\xi (\mathcal{H}_k(M)))\neq 0, 
  \end{eqnarray}
  then $N\geq t+k$; 
  \item
  Let  $f: M\longrightarrow \mathbb{C}^N$  be  a  complex $k$-regular embedding. If 
  there  exists  a  $k$-uniform  hypergraph  $\mathcal{H}_k(M)$  on  $M$  such  that
   $\vec{\mathcal{H}}_k(M)=\pi^{-1}(\mathcal{H}_k(M))$  and 
   the  $t$-th   dual   Chern  class
   \begin{eqnarray}\label{eq-intro-92}
   \bar  c_t(\xi (\mathcal{H}_k(M)))\neq 0, 
   \end{eqnarray}
   then $N\geq t+k$.    
  \end{enumerate}
 \end{proposition}
 
 \begin{proof}
 (1)   The  canonical  inclusion  $\varphi:  \mathcal{H}_k(M)\longrightarrow  {\rm  Conf}_k(M)/\Sigma_k$  
 induces  a  pull-back  of  vector  bundles  
 \begin{eqnarray*}
  \varphi^*(\xi^k(M))  =  \xi (\mathcal{H}_k(M)).  
 \end{eqnarray*}
 Thus  (\ref{eq-intro-91})  implies    (\ref{eq-intro-97}).  
 By  Lemma~\ref{th-intro-1}~(1),  we  obtain   $N\geq t+k$.   
 
 (2)   The  canonical  inclusion  $\varphi:  \mathcal{H}_k(M)\longrightarrow  {\rm  Conf}_k(M)/\Sigma_k$  
 induces  a  pull-back  of  vector  bundles  
 \begin{eqnarray*}
  \varphi^*(\xi^k(M)\otimes \mathbb{C})  =  \xi (\mathcal{H}_k(M))\otimes \mathbb{C}.  
 \end{eqnarray*}
 Thus  (\ref{eq-intro-92})  implies    (\ref{eq-intro-98}).  
 By  Lemma~\ref{th-intro-1}~(2),  we  obtain   $N\geq t+k$. 
 \end{proof}

\section{Double  complexes  for       configuration  spaces}\label{sec6}

In  this  section,  we  construct  double  complexes  for  configuration  spaces  by  
using the  double  complexes   for   $\Delta$-manifolds  in  Subsection~\ref{ssec-3.1}.

\begin{lemma}\label{le-8.2}
The  configuration  spaces    ${\rm  Conf}_\bullet(M)= ({\rm  Conf}_{n+1}(M))_{n\in\mathbb{N}}$
  is  a  $\Delta$-manifold  
with  the  face  maps  
\begin{eqnarray}\label{eq-7.a3}
\partial^n_i:   {\rm  Conf}_{n+1}(M)\longrightarrow  {\rm   Conf}_{n}(M) 
\end{eqnarray} 
sending  $(x_1,\ldots,  x_{n+1})$  to  $(x_1,\ldots,\widehat{x_{i+1}},  \ldots,  x_{n+1})$  by  removing the  $i$-th  coordinate.  
\end{lemma}

\begin{proof}
It suffices  to  prove  the   $\Delta$-identity
 for  any  $n\in  \mathbb{N}$  and    any   $i,j\in\mathbb{N}$  such  that     $0\leq  i<j\leq  n$.     
 Let  $(x_1,\ldots,  x_{n+1})\in   {\rm  Conf}_{n+1}(M)$.   Then  
 \begin {eqnarray*}
  \partial^{n-1}_i \partial^n_j  (x_1,\ldots,  x_{n+1}) &=&    \partial^{n-1}_i ( x_1,\ldots, \widehat{x_{j+1}},
  \ldots,  x_{n+1})\\
  &=&  ( x_1,\ldots, \widehat{x_{i+1}},  \ldots,  \widehat{x_{j+1}},
  \ldots,  x_{n+1})\\
  &=&\partial ^{n-1}_{j-1}  ( x_1,\ldots, \widehat{x_{i+1}},  \ldots,    x_{n+1})\\
  &=&  \partial ^{n-1}_{j-1} \partial  ^n_i   (x_1,\ldots,  x_{n+1}).  
    \end{eqnarray*}
Therefore,   $\partial^{n-1}_i \partial^n_j= \partial ^{n-1}_{j-1} \partial  ^n_i$ 
and  consequently  ${\rm  Conf}_\bullet(M) $
  is  a  $\Delta$-manifold.    
\end{proof}

\begin{theorem}\label{pr-ab1}
$(\Omega^\bullet({\rm  Conf}_\bullet(M)),  d,\partial)$  is  a  double  complex  with   
a  graded  group  action  of  $\Sigma_n$  on  $\Omega^\bullet({\rm  Conf}_n(M))$  for  each 
$n\geq  1$  which  is  commutative  with  $d$.  
\end{theorem}

\begin{proof}
By   Lemma~\ref{le-2.1}  and   Lemma~\ref{le-8.2},  we   obtain  a  double  complex
  $(\Omega^\bullet({\rm  Conf}_\bullet(M)),  d,\partial)$.  
  For  each  $n\geq  1$,  consider  the  group  action  of  
  $\Sigma_n$  on   ${\rm  Conf}_n(M)$  sending  $(x_1,\ldots,x_n)$  to  
  $(x_{s(1)},\ldots, x_{s(n)})$  for  any  $s\in \Sigma_n$.  
  The  induced  pull-back    of  differential  forms 
  \begin{eqnarray*}
  s^*:  \Omega^\bullet({\rm  Conf}_n(M))\longrightarrow  \Omega^\bullet({\rm  Conf}_n(M))   
  \end{eqnarray*}
  gives  a   group  action  of  $\Sigma_n$  on  $\Omega^\bullet({\rm  Conf}_n(M))$
  such  that  $(s_1\circ  s_2)^*=  s_2^* \circ  s_1^*$
   for  any  $s_1,s_2\in\Sigma_n$.    
 The  pull-back  maps $s^*$  of  differential  forms  commute  with  the  exterior  derivative,  i.e.    $s^*  \circ  d =   d\circ  s^*$  for  any  $s\in  \Sigma_n$.   
 Thus   the  group  action  of  $\Sigma_n$  on  $\Omega^\bullet({\rm  Conf}_n(M))$    is  commutative  with  $d$.  
\end{proof}

\begin{corollary}\label{co-mzoa1}
For  any  $r\geq 0$,  we  have  a  double  complex  $(\Omega^\bullet({\rm  Conf}_\bullet(M,r)),  d,\partial)$,  which  is  a   quotient  double  complex  of  
 $(\Omega^\bullet({\rm  Conf}_\bullet(M)),  d,\partial)$,    
     with   
a  graded  group  action  of  $\Sigma_n$  on  $\Omega^\bullet({\rm  Conf}_n(M,r))$     commutative  with  $d$  for  each 
$n\geq  1$.  
Moreover, let   
\begin{eqnarray}\label{eq-baq1}
\iota_r:  {\rm  Conf}_\bullet(M,r)\longrightarrow  {\rm  Conf}_\bullet(M)
\end{eqnarray}
  be  the  canonical  inclusion.  
  Then  
$({\rm   Ker}(\iota_r^\#))_{r\geq  0}$  is  a  family  of  double  complexes 
  giving  a  filtration  of   
 $(\Omega^\bullet({\rm  Conf}_\bullet(M)),  d,\partial)$.   
\end{corollary}

\begin{proof}
Note  that         $ {\rm  Conf}_\bullet(M,r)$   is   
a   $\Delta$-submanifold  of       $ {\rm  Conf}_\bullet(M )$.  
 By  Lemma~\ref{le-6-987},  the  canonical  inclusion  
 (\ref{eq-baq1})    
 induces  a  surjective  chain    map   
  \begin{eqnarray*}
  \iota_r^\#:  (\Omega^\bullet({\rm  Conf}_\bullet(M)),  d,\partial)\longrightarrow  (\Omega^\bullet({\rm  Conf}_\bullet(M,r)),  d,\partial)
\end{eqnarray*}
of  double  complexes  and  
\begin{eqnarray*}
 (\Omega^\bullet({\rm  Conf}_\bullet(M,r)),  d,\partial)\cong  
 \frac{ (\Omega^\bullet({\rm  Conf}_\bullet(M)),  d,\partial)}{{\rm   Ker}  (   \iota_r^\#) }  
\end{eqnarray*}
is  a   quotient  double  complex  of  
 $(\Omega^\bullet({\rm  Conf}_\bullet(M)),  d,\partial)$.   
   With  the  help  of  Example~\ref{ex-2.9vz}~(4),   since  $ {\rm  Conf}_n(M,r)$  is  a  $\Sigma_n$-invariant  submanifold  of  
$ {\rm  Conf}_n(M)$   for  each  $n\geq  1$,  the  induced  pull-backs  of  differential  forms  
   give  a group  action  of  $\Sigma_n$  on  $\Omega^\bullet({\rm  Conf}_n(M,r))$     commutative  with  $d$.    Moreover,  with  the  help  of  Example~\ref{ex-2.9vz}~(1)  -  (3)  respectively,  
 we  have  that  the  family  $({\rm   Ker}(\iota_r^\#))_{r\geq  0}$     of  double  complexes  
 satisfies  
 \begin{enumerate}[(1)]
  \item
    $ {\rm   Ker}  (   \iota_0^\#) = 0$;   
        \item   
  $ {\rm   Ker}  (   \iota_r^\#)  \subseteq   {\rm   Ker}  (   \iota_s^\#) $        for  any   $0\leq  r  <s$; 
  \item
  if  $M$  is  path-connected,  then   
$
\lim_{r\to+\infty}   {\rm   Ker}  (   \iota_r^\#)  =   (\Omega^\bullet({\rm  Conf}_\bullet(M)),  d,\partial)$.   
  \end{enumerate}
  Thus   $({\rm   Ker}(\iota_r^\#))_{r\geq  0}$  
  gives  a  filtration  of   
 $(\Omega^\bullet({\rm  Conf}_\bullet(M)),  d,\partial)$.
\end{proof}

A  {\it  continuous  total  order}  on  $M$  
is  a  total  order  $\prec  $   on  $M$  such  that  for  any  $p\prec  q$,  $p,q\in  M$,  
there  exists     open   neighborhoods  $U_p$  of  $p$  and  $U_q$  of  $q$  such  that 
$p'\prec  q'$  for  any  $p'\in  U_p$  and  any  $q'\in  U_q$.  
For examples,   
\begin{enumerate}[(1)]
\item
  let  $M=V$.   Then  any  total  order  on  $V$  is  a  continuous   total  order;    
 \item
  let  $M=\mathbb{R}$.  Then   the  canonical  order   on  $\mathbb{R}$  is  a  continuous  total  order;
  \item
  let $M=\bigcup_{i\in  V}  \mathbb{R}_i$  be  a  disjoint  union  of   real  lines  indexed  by  $V$.  
  Then  a  total  order  on  $V$  and 
   a  continuous  total  order  on     $\mathbb{R}_i$  for  each  $i\in V$  
  give  a  continuous  total  order  on  $M$.  
\end{enumerate}
Conversely,  if  $M$  has  a  continuous  total  order  indexed  by  
a  subset  of  $\mathbb{R}$,  then  
$M$  can  be  continuously  embedded  into  $\mathbb{R}$  hence  $M$  must be  one  
of  the  above  three  examples.

\begin{lemma}\label{le-8.222}
Suppose  $M$  has  a  fixed   continuous  total  order  $\prec$.  
Then  the  configuration  spaces    ${\rm  Conf}_\bullet(M)/\Sigma_\bullet= ({\rm  Conf}_{n+1}(M)/\Sigma_{n+1})_{n\in\mathbb{N}}$
  is  a  $\Delta$-manifold  
with  the  face  maps  
\begin{eqnarray}\label{eq-7.a93}
\partial^n_i:   {\rm  Conf}_{n+1}(M)/\Sigma_{n+1}\longrightarrow  {\rm   Conf}_{n}(M) /\Sigma_n
\end{eqnarray} 
sending  $\{x_1,\ldots,  x_{n+1}\}$  to  $\{x_1,\ldots,\widehat{x_{i+1}},  \ldots,  x_{n+1}\}$  by  removing the  $i$-th  coordinate,  where  $x_1\prec  \cdots\prec  x_{n+1}$.  
\end{lemma}

\begin{proof}
The  proof  is  similar   with  Lemma~\ref{le-8.2}.  
Let  $\{x_1,\ldots,  x_{n+1}\}\in   {\rm  Conf}_{n+1}(M)/\Sigma_{n+1}$
 such  that    $x_1\prec  \cdots\prec  x_{n+1}$.     Then    for  any   $0\leq  i<j\leq  n$,  
 \begin {eqnarray*}
  \partial^{n-1}_i \partial^n_j  \{x_1,\ldots,  x_{n+1}\}  
  &=& \{ x_1,\ldots, \widehat{x_{i+1}},  \ldots,  \widehat{x_{j+1}},
  \ldots,  x_{n+1}\}\\
 & = &    
      \partial ^{n-1}_{j-1} \partial  ^n_i   \{x_1,\ldots,  x_{n+1}\}.  
    \end{eqnarray*}
Therefore,   $\partial^{n-1}_i \partial^n_j= \partial ^{n-1}_{j-1} \partial  ^n_i$.  
Since  the  order  $\prec$  on  $M$  is  continuous, 
    the  face  maps  $\partial^n_i$     are   smooth  maps
     for  $1\leq  i\leq  n$.  Consequently,  ${\rm  Conf}_\bullet(M) /\Sigma_\bullet$
  is  a  $\Delta$-manifold.    
\end{proof}

\begin{theorem}\label{pr-ab122222}
Suppose  $M$  has  a  fixed   continuous  total  order  $\prec$.  
Then  
\begin{eqnarray}\label{eq-abz}
(\Omega^\bullet({\rm  Conf}_\bullet(M)/\Sigma_\bullet),  d,\partial) 
\end{eqnarray}
 a   sub-double  complex  of  $(\Omega^\bullet({\rm  Conf}_\bullet(M)),  d,\partial)$  consisting  of  the  $\Sigma_\bullet$-invariant  elements.   
\end{theorem}

\begin{proof}
By   Lemma~\ref{le-2.1}  and   Lemma~\ref{le-8.222},  we   obtain  a  double  complex
  (\ref{eq-abz}).  Let  $n\geq 1$.  The  covering  map  (\ref{eq-0.a1})  induces  a  pull-back     of  differential  forms 
  \begin{eqnarray}\label{eq-cht}
  \pi_n^*:  \Omega^\bullet({\rm  Conf}_n(M)/\Sigma_n)\longrightarrow  \Omega^\bullet({\rm  Conf}_n(M)).  
  \end{eqnarray}
  Since  $\pi_n$  is  surjective,  the   induced  map    
  (\ref{eq-cht})  is  injective.  Moreover,  the  image  of   (\ref{eq-cht})   is  the  space  of  the  differential  forms  on  ${\rm  Conf}_n(M) $   that  is  
   invariant under  the   permutation   of  $(x_1,\ldots,x_n)$  in  (\ref{eq-00bz}).  
   Therefore,      $\Omega^\bullet ({\rm  Conf}_n(M)/\Sigma_n)$   in  (\ref{eq-abz})   of  all  the  differential  forms  on   ${\rm  Conf}_n(M)/\Sigma_n $  
    is  isomorphic  to  the  space  of  the  differential  forms  on  ${\rm  Conf}_n(M) $   that  is  
   invariant under  the   permutation   of   $(x_1,\ldots,x_n)$  in  (\ref{eq-00bz}).
    Consequently,  (\ref{eq-abz})  a   sub-double  complex  of  $(\Omega^\bullet({\rm  Conf}_\bullet(M)),  d,\partial)$  consisting  of  the  $\Sigma_n$-invariant  elements  for  any  $n\geq  1$.   
   \end{proof}

\begin{corollary}\label{co-mzoa2}
Suppose  $M$  has  a  fixed   continuous  total  order  $\prec$.  
Then  for  any  $r\geq 0$,  we  have     $(\Omega^\bullet({\rm  Conf}_\bullet(M,r)/\Sigma_\bullet),  d,\partial)$  as   a   quotient  double  complex  of  
 $(\Omega^\bullet({\rm  Conf}_\bullet(M)/\Sigma_\bullet),  d,\partial)$.      
  Moreover,  let   
  \begin{eqnarray}\label{eq-baq2}
  \iota_r/\Sigma_\bullet: {\rm  Conf}_\bullet(M,r)/\Sigma_\bullet\longrightarrow   {\rm  Conf}_\bullet(M)/\Sigma_\bullet
  \end{eqnarray}
  be  the  canonical  inclusion.  Then  
$({\rm   Ker}((\iota_r/\Sigma_\bullet)^\#))_{r\geq  0}$  is  a  family  of  double  complexes 
  giving  a  filtration  of   
 $(\Omega^\bullet({\rm  Conf}_\bullet(M)/\Sigma_\bullet),  d,\partial)$.   
\end{corollary}

\begin{proof}
The  proof  is  similar with  Corollary~\ref{co-mzoa1}.  
Since  $M$  has  a  continuous total  order,  
$ {\rm  Conf}_\bullet(M,r)/\Sigma_\bullet$   is   
a   $\Delta$-submanifold  of       $ {\rm  Conf}_\bullet(M )/\Sigma_\bullet$.  
By  Lemma~\ref{le-6-987},  the  canonical  inclusion  
 (\ref{eq-baq2})    
 induces  a  surjective  chain    map   
  \begin{eqnarray*}
 ( \iota_r/\Sigma_\bullet)^{\#}:  (\Omega^\bullet({\rm  Conf}_\bullet(M)/\Sigma_\bullet),  d,\partial)\longrightarrow  (\Omega^\bullet({\rm  Conf}_\bullet(M,r)/\Sigma_\bullet),  d,\partial)
\end{eqnarray*}
of  double  complexes  and  
\begin{eqnarray*}
 (\Omega^\bullet({\rm  Conf}_\bullet(M,r)/\Sigma_\bullet),  d,\partial)\cong 
 \frac{  (\Omega^\bullet({\rm  Conf}_\bullet(M)/\Sigma_\bullet),  d,\partial)}{{\rm   Ker}  (   (\iota_r/\Sigma_\bullet)^\#)}   
\end{eqnarray*}
is  a   quotient  double  complex  of  
 $(\Omega^\bullet({\rm  Conf}_\bullet(M)/\Sigma_\bullet),  d,\partial)$.   
With  the  help  of  Example~\ref{ex-2.9vz}~(1)  -  (3)  respectively,  
 we  have  that  the  family  $({\rm   Ker}(( \iota_r/\Sigma_\bullet)^{\#}))_{r\geq  0}$     of  double  complexes  
 satisfies  
 \begin{enumerate}[(1)]
  \item
    $ {\rm   Ker}  (   ( \iota_0/\Sigma_\bullet)^{\#}) = 0$;   
        \item   
  $ {\rm   Ker}  (   ( \iota_r/\Sigma_\bullet)^{\#})  \subseteq   {\rm   Ker}  (  ( \iota_s/\Sigma_\bullet)^{\#}) $        for  any   $0\leq  r  <s$; 
  \item
  if  $M$  is  path-connected,  then   
$
\lim_{r\to+\infty}   {\rm   Ker}  (   ( \iota_r/\Sigma_\bullet)^{\#})  =   (\Omega^\bullet({\rm  Conf}_\bullet(M)/\Sigma_\bullet),  d,\partial)$.   
  \end{enumerate}
  Thus   $({\rm   Ker}( ( \iota_r/\Sigma_\bullet)^{\#})_{r\geq  0}$  
  gives  a  filtration  of   
 $(\Omega^\bullet({\rm  Conf}_\bullet(M)/\Sigma_\bullet),  d,\partial)$.  
\end{proof}

\begin{remark}
If  $M$  is  assumed to  have a  continuous  total  order
indexed  by  a  subset  of  $\mathbb{R}$,    then  we  have  
$\dim  M\leq  1$  thus  
the  double  complexes  in  Theorem~\ref{pr-ab1}  and  
Corollary~\ref{co-mzoa2}  only have  $0$-forms and  $1$-forms.  
\end{remark}

\section{Double  complexes  for  hypergraphs}\label{sec7}

In  this  section,  we  construct  double  complexes  for  hypergraphs  with  vertices  
from  manifolds  by  
using the  double  complexes   for   graded   submanifolds  of  $\Delta$-manifolds 
  in  Subsection~\ref{ss2.3a}   as  well  as  the  double  complexes  for  associated  $\Delta$-manifolds   in  Subsection~\ref{ssec-3.3}.

Let  $\vec{\mathcal{H}}(M)$  be  a   hyperdigraph  on  $M$  given by  a  union  of      uniform  hyperdigraphs 
 \begin{eqnarray*}
\vec{\mathcal{H}}(M)= \bigcup_{n=1}^\infty\vec{\mathcal{H}}_n(M).  
 \end{eqnarray*}
 Then   $\vec{\mathcal{H}}_n(M)$  is  a  submanifold  of  ${\rm  Conf}_n(M)$  for each  $n\geq 1$.  
  By  Definition~\ref{def-2.3mko},  
   \begin{eqnarray*}
   \vec{\mathcal{H}}(M)=(\vec{\mathcal{H}}_{n}(M))_{n\geq  1} 
   \end{eqnarray*}
      is  a  graded  submanifold  of  the  $\Delta$-manifold  
  ${\rm  Conf}_\bullet(M)$.  
  Let  $\epsilon_n$  be  the  canonical  embedding  
  of   $\vec{\mathcal{H}}_n(M)$  into  ${\rm  Conf}_n(M)$.  
  If  we  let   $A_\bullet={\rm Conf}_\bullet(M)$  and  $B_\bullet= \vec{\mathcal{H}} (M)$
   in  (\ref{eq-mm1n})  and (\ref{eq-mm2n})  respectively,  then  we  
   have  the  infimum  chain  complex  as  well  as  the  supremum  chain  complex  
  \begin{eqnarray*}
   {\rm  Inf}^\bullet(\Omega^\bullet(\vec{\mathcal{H}} (M))),  
  ~~~~~~  
     {\rm  Sup}^\bullet(\Omega^\bullet(\vec{\mathcal{H}} (M))).       
  \end{eqnarray*}
  \begin{theorem}
  \label{th-main1}
  We  have  
  \begin{enumerate}[(1)]
  \item
   sub-double  complexes  
    \begin{eqnarray}\label{eq-lzpaqgj98}
 ( {\rm  Inf}_\bullet(\Omega^\bullet(\epsilon_\bullet),  d,   \partial ), ~~~~~~   ( {\rm  Sup}_\bullet(\Omega^\bullet(\epsilon_\bullet),  d,   \partial )
  \end{eqnarray}
  of    $(\Omega^\bullet({\rm  Conf}_\bullet(M)),  d,\partial)$ 
  such  that  the  canonical  inclusion  
\begin{eqnarray*}
\theta:    {\rm  Inf}_\bullet(\Omega^\bullet(\epsilon_\bullet))
\longrightarrow  {\rm  Sup}_\bullet(\Omega^\bullet(\epsilon_\bullet))    \end{eqnarray*}
 is  a  quasi-isomorphism  with respect  to  $\partial$; 
 \item
    quotient  double  complexes  
     \begin{eqnarray}\label{eq-aab}
 ( {\rm  Inf}^\bullet(\Omega^\bullet(\vec{\mathcal{H}} (M))),  d,   \partial ), ~~~~~~   ( {\rm  Sup}^\bullet(\Omega^\bullet(\vec{\mathcal{H}} (M))),  d,   \partial )
  \end{eqnarray}
of   $(\Omega^\bullet({\rm  Conf}_\bullet(M)),  d,\partial)$     such  that  the  canonical  quotient map  
\begin{eqnarray*}
q:  {\rm  Sup}^\bullet(\Omega^\bullet(\vec{\mathcal{H}} (M)))\longrightarrow  {\rm  Inf}^\bullet(\Omega^\bullet(\vec{\mathcal{H}} (M)))
\end{eqnarray*}
 is  a  quasi-isomorphism  with respect  to $\partial$.  
 \end{enumerate} 
  \end{theorem}

  \begin{proof}
  In   Lemma~\ref{le-1.1},  we  let  $A_\bullet$  be  the  
  $\Delta$-manifold  ${\rm Conf}_\bullet(M)$ 
  and  let  $B_\bullet$ be   the  graded  submanifold  
  $\vec{\mathcal{H}}(M)$.  
  Then we  obtain  the  theorem.  
  \end{proof}

 \begin{theorem}
 \label{th-main1aa}
 For  any   hyperdigraph  $\vec{\mathcal{H}}(M)$  on  $M$,  
 we have  surjective  homomorphisms  of  double  complexes 
 \begin{eqnarray}\label{eq-9.laqgh}
&&(\Omega^\bullet( \Delta \vec{\mathcal{H}}(M)),  d,\partial)\longrightarrow   ({\rm  Sup}^\bullet(\vec{\mathcal{H}} (M)), d,\partial)
\nonumber\\
&&\longrightarrow ( {\rm  Inf}^\bullet(\vec{\mathcal{H}} (M)),   d,\partial)
\longrightarrow  (\Omega^\bullet( \delta \vec{\mathcal{H}}(M)),  d,\partial).  
 \end{eqnarray}
 Moreover,  $  \vec{\mathcal{H}}(M) $  is  a  $\Delta$-submanifold  of  ${\rm   Conf}_\bullet(M)$
  iff  all   the three  homomorphisms in  (\ref{eq-9.laqgh})  are  the  identity.  
 \end{theorem}
 
 \begin{proof}
 In   Lemma~\ref{le-sqc},   we   let  
 $A_\bullet$  be  the  $\Delta$-manifold   ${\rm   Conf}_\bullet(M)$  and  let   
 $B_\bullet$  be  the  graded  submanifold  $\vec{\mathcal{H}}(M) $.   
We  obtain  the  theorem.   
 \end{proof}

  \begin{example}
  By  Theorem~\ref{pr-ab1},  we  have  a  $\Sigma_n$-action  on  $(\Omega^\bullet({\rm  Conf}_\bullet(M)),  d,\partial)$  commutative  with  $d$.  
  Suppose  $\vec{\mathcal{H}}_n(M)$  is  $\Sigma_n$-invariant  for  each  $n\geq  1$.  
  That   is,   we  have  
  \begin{eqnarray*}
  s(\vec{\sigma}(M))\in  \vec{\mathcal{H}}_n(M) 
  \end{eqnarray*}  
  for  any  $\vec{\sigma}(M)\in \vec{\mathcal{H}}_n(M)$  and  any  $s\in \Sigma_n$.  
  Then  the  $\Sigma_n$-action  on  $(\Omega^\bullet({\rm  Conf}_\bullet(M)),  d,\partial)$  
  induces  a   $\Sigma_n$-action  on   
  the  two   sub-double     complexes  in  (\ref{eq-lzpaqgj98})  
  as  well as  the  two   quotient  double    complexes  in  
  (\ref{eq-aab}).

  In  particular, 
   let  $\vec{\mathcal{H}}_\bullet (M)$  be  ${\rm  Conf}_\bullet(M,r)$  for  $r\geq  0$.  
  Then  by  Corollary~\ref{le-2.99a},  both  of   the   double     complexes in  (\ref{eq-lzpaqgj98}) 
  are   equal to  
\begin{eqnarray*}
( {\rm  Ker} (\iota_r^\#),   d,  \partial)
\end{eqnarray*}
where  $\iota_r$  is  given  by (\ref{eq-baq1})  and  
   both  of   the   the   double     complexes in  (\ref{eq-aab})  are   equal to  
      \begin{eqnarray*}
      (\Omega^\bullet({\rm  Conf}_\bullet(M,r)),  d,\partial). 
      \end{eqnarray*}
  \end{example}

  Suppose  $M$  has  a  fixed  continuous  total  order  $\prec$.  
  Then  ${\rm  Conf}_\bullet (M)/\Sigma_\bullet$  is  a  $\Delta$-manifold.  
 Let  
 \begin{eqnarray*}
 \mathcal{H}_n(M)=\pi_n(\vec{\mathcal{H}}_n(M)) 
 \end{eqnarray*}
   for  each  $n\geq  1$  and  let
 \begin{eqnarray*}
 \mathcal{H}(M)=\bigcup_{n\geq  1}    \mathcal{H}_n(M) 
 \end{eqnarray*}
   be  a  hypergraph  on  $M$.  
 Then   $\mathcal{H}(M)$  is  a  graded  submanifold  of  ${\rm  Conf}_\bullet (M)/\Sigma_\bullet$.  
 Let  $\epsilon_n/\Sigma_n$  be  the  canonical  embedding  of  
 $\mathcal{H}_n(M)$  into  ${\rm  Conf}_n(M)/\Sigma_n$  for each  $n\geq  1$.  
   If  we  let   $A_\bullet={\rm Conf}_\bullet(M)/\Sigma_\bullet$  and  $B_\bullet= {\mathcal{H}} (M)$
  in  Subsection~\ref{ss2.3a},  then  we  
   have  the  infimum  chain  complex  as  well  as  the  supremum  chain  complex   
  \begin{eqnarray*}
  {\rm  Inf}^\bullet(\Omega^\bullet({\mathcal{H}} (M))),  
  ~~~~~~
 {\rm  Sup}^\bullet(\Omega^\bullet({\mathcal{H}} (M)) )      
  \end{eqnarray*}
   for  (\ref{eq-mm1n})  and (\ref{eq-mm2n})  respectively.  

  \begin{theorem}
  \label{th-main2}
   Suppose  $M$  has  a  fixed  continuous  total  order  $\prec$.  
  Then  we  have  
  \begin{enumerate}[(1)]
  \item
   sub-double  complexes  
    \begin{eqnarray}\label{eq-lzpaqgj9822}
 ( {\rm  Inf}_\bullet(\Omega^\bullet(\epsilon_\bullet/\Sigma_\bullet),  d,   \partial ), ~~~~~~   ( {\rm  Sup}_\bullet(\Omega^\bullet(\epsilon_\bullet/\Sigma_\bullet),  d,   \partial )
  \end{eqnarray}
  of    $(\Omega^\bullet({\rm  Conf}_\bullet(M)/\Sigma_\bullet),  d,\partial)$ 
  such  that  the  canonical  inclusion  
\begin{eqnarray*}
\theta:    {\rm  Inf}_\bullet(\Omega^\bullet(\epsilon_\bullet/\Sigma_\bullet))
\longrightarrow  {\rm  Sup}_\bullet(\Omega^\bullet(\epsilon_\bullet/\Sigma_\bullet))  
  \end{eqnarray*}
 is  a  quasi-isomorphism  with respect  to  $\partial$;   
 \item
    quotient  double  complexes  
     \begin{eqnarray}\label{eq-aab22}
 ( {\rm  Inf}^\bullet({\mathcal{H}} (M)),  d,   \partial ), ~~~~~~   ( {\rm  Sup}^\bullet({\mathcal{H}} (M)),  d,   \partial )
  \end{eqnarray}
of   $(\Omega^\bullet({\rm  Conf}_\bullet(M)/\Sigma_\bullet),  d,\partial)$     such  that  the  canonical  quotient map  
\begin{eqnarray*}
q:  {\rm  Sup}^\bullet({\mathcal{H}} (M))\longrightarrow  {\rm  Inf}^\bullet({\mathcal{H}} (M))
\end{eqnarray*}
 is  a  quasi-isomorphism  with respect  to $\partial$.  
 \end{enumerate} 
  \end{theorem}

  \begin{proof}
  In   Lemma~\ref{le-1.1},  we  substitute  $A_\bullet$  with  the  
  $\Delta$-manifold  ${\rm Conf}_\bullet(M)/\Sigma_\bullet$ 
  and  substitute  $B_\bullet$ with  the  graded  submanifold  
  ${\mathcal{H}}(M)$.  
  Then we  obtain  the  theorem.  
  \end{proof}

 \begin{theorem}
 \label{th-main1bb}
  Suppose  $M$  has  a  fixed  continuous  total  order  $\prec$. 
 For  any     hypergraph  $ {\mathcal{H}}(M)$  on  $M$,  
 we have  surjective  homomorphisms  of  double  complexes 
 \begin{eqnarray}\label{eq-77.laqgh}
&&(\Omega^\bullet( \Delta  {\mathcal{H}}(M)),  d,\partial)\longrightarrow   ({\rm  Sup}^\bullet( {\mathcal{H}} (M)), d,\partial)
\nonumber\\
&&\longrightarrow ( {\rm  Inf}^\bullet( {\mathcal{H}} (M)),   d,\partial)
\longrightarrow  (\Omega^\bullet( \delta {\mathcal{H}}(M)),  d,\partial).  
 \end{eqnarray}
 Moreover,  $   {\mathcal{H}}(M) $  is  a  $\Delta$-submanifold  of  ${\rm   Conf}_\bullet(M)/\Sigma_\bullet$
  iff  all   the three  homomorphisms in  (\ref{eq-77.laqgh})  are  the  identity.  
 \end{theorem}
 
 \begin{proof}
 In   Lemma~\ref{le-sqc},   we   let  
 $A_\bullet$  be  the  $\Delta$-manifold   ${\rm   Conf}_\bullet(M)/\Sigma_\bullet$  and  let   
 $B_\bullet$  be  the  graded  submanifold  ${\mathcal{H}}(M) $.   
We  obtain  the  theorem.   
 \end{proof}

\begin{remark}
If  $M$  is  assumed to  have a  continuous  total  order
indexed  by  a  subset  of  $\mathbb{R}$,    then  
the  double  complexes  in  Theorem~\ref{th-main2}  and  
Theorem~\ref{th-main1bb}   only have  $0$-forms and  $1$-forms.  
\end{remark}


\begin{example}\label{ex-7.9}
Let  $M=V$  be  a  discrete  set.   Then  in   Theorem~\ref{th-main2}  and  Theorem~\ref{th-main1bb},  
\begin{enumerate}[(1)]
\item
  the  exterior  derivative  $d$  is  trivial  hence  all  the  double  complexes  in  
(\ref{eq-lzpaqgj9822}),   (\ref{eq-aab22})   and   (\ref{eq-77.laqgh}) 
  reduce  to   chain  complexes  with  boundary  map 
$\partial$;  
\item
the  two  complexes  in  (\ref{eq-aab22})   can  be  respectively   
obtained  from   the  infimum  chain  complex  ${\rm  Inf}_\bullet(\mathcal{H})$  and  
the  supremum  chain  complex    ${\rm  Sup}_\bullet(\mathcal{H})$, 
which  are    defined  in  \cite[Section~3.1]{h1},   
by  applying the  Hom  functor; 
\item
a  hypergraph  $\mathcal{H}$  on  $V$  is  a  $\Delta$-submanifold  of  
${\rm  Conf}_\bullet(V)/\Sigma_\bullet$  iff  
$\mathcal{H}$  is  a  simplicial complex.  
\end{enumerate}
\end{example}

 \section*{Acknowledgement}   The   author  would like to express his  deep
gratitude to the referee for the  careful reading of the manuscript.

    \bigskip

Shiquan Ren

Address:
School  of  Mathematics and Statistics,  Henan University,  Kaifeng   475004,  China.

e-mail:  renshiquan@henu.edu.cn


\begin{thebibliography}{99}




\bibitem{confcomp}
H. Alpert, M.  Kahle  and  R. MacPherson,  \emph{Configuration spaces of disks in an infinite strip. }
 J.  Appl.   Comput.  Topol.   {\bf  5}  (2021),   357-390.  
 


\bibitem{gt}
H.   Alpert  and  Fedor Manin,   
\emph{Configuration spaces of disks in a strip, twisted algebras, persistence, and other stories}.  
Geom. Topol.  {\bf  28}  (2024),   641-699. 



\bibitem{imrn1}
Y.  Baryshnikov, P.   Bubenik  and  M. Kahle,  \emph{Min-type Morse theory for configuration spaces of hard spheres. } Int.  Math. Res. Notices {\bf  9}  (2014), 2577-2592. 



 
\bibitem{berge}
 C. Berge,   \emph{Graphs and hypergraphs.} North-Holland Mathematical Library, Amsterdam, 1973.



\bibitem{lap1}
 M.   Belkin and P. Niyogi,  \emph{ Laplacian eigenmaps for dimensionality reduction and data representation.}   Neural Comput.  {\bf   15}  (2003),  1373-1396.
 
\bibitem{lap2}
 M.  Belkin and P.  Niyogi,  \emph{Towards a theoretical foundation for laplacian-based manifold methods.}    International Conference on Computational Learning Theory  (2005),   486-500.






\bibitem{siam1}
C.  Bick,  E.  Gross,  H.  A.  Harrington  and  M.  T. Schaub, 
\emph{What are higher-order networks?  } SIAM Rev. {\bf 65}(3) (2023),  686-731.



   

\bibitem{high1}
P. Blagojevi$\check{\text{c}}$,   W. L{\"u}ck and G. Ziegler, \emph{On highly regular embeddings}.  Trans. Amer. Math. Soc.  {\bf  368}(4)   (2016),   2891-2912.  


\bibitem{high2}
P. Blagojevi$\check{\text{c}}$,  F. Cohen, W. L{\"u}ck and G. Ziegler,
\emph{On complex highly regular embeddings and the extended Vassiliev conjecture}.
 Int. Math. Res. Not. {\bf 2016}(20)  (2016),   6151-6199. 
 


\bibitem{homol}
C.-F. B\"odigheimer,   F. Cohen   and L. Taylor,  \emph{On the homology of configuration spaces. }
Topology  {\bf  28}   (1989),   111-123.  
 
\bibitem{homol2}
C.-F. B\"odigheimer and F. R. Cohen, \emph{Rational cohomology of configuration spaces of surfaces.}  Algebraic topology and transformation groups (G\"ottingen 1987), Lecture Notes in Math. 1361, Springer, Berlin 1988, 7-13.



\bibitem{Borsuk}
K. Borsuk, \emph{On the $k$-independent subsets of the Euclidean space and of the Hilbert space}.  Bull. Acad. Polon. Sci. Cl. III. {\bf{5}} (1957), 351-356. 



\bibitem{h1}
S.  Bressan,  J.  Li,  S. Ren  and  J. Wu,  \emph{The  embedded homology of hypergraphs and applications.}    Asian J.    Math. {\bf 23}(3)   (2019),  479-500.


 \bibitem{phys}
G.  Carlsson, J.  Gorham, M.  Kahle   and J.  Mason,  \emph{Computational topology for configuration spaces of hard disks. }  Phys. Review  E {\bf  85}  (2012),  011303.  


\bibitem{hdg}
D. Chen,    J.  Liu,  
  J. Wu   and G.-W.  Wei,  \emph{Persistent hyperdigraph homology and persistent hyperdigraph Laplacians. }   Found.   Data Sci.  
{\bf 5}   (2023),  558-588. 




\bibitem{chung}
F. R. K. Chung and R. L. Graham,  \emph{Cohomological aspects of  hypergraphs.} 
Trans.   Amer.  Math.  Soc. 
{\bf  334} (1992),  365-388.

\bibitem{inv}
T.  Church,  \emph{Homological stability for configuration spaces of manifolds}.  
Invent. Math. {\bf  188}  (2012),  465-504.

\bibitem{duke}
T.  Church, J.  S. Ellenberg   and   B.  Farb,  \emph{FI-modules and stability for representations of symmetric groups}.  Duke Math. J. {\bf  164}  (2015),  1833-1910.  


\bibitem{bundle1983}
F.R. Cohen, R.L. Cohen, N.J. Kuhn and J.L. Neisendorfer, \emph{Bundles over configuration spaces.}  { Pacific J. Math. }{\bf{104}} (1983), 47-54.



\bibitem{bundle1989}
F.R. Cohen, R.L. Cohen, B. Mann and  R.J. Milgram, \emph{Divisors and configurations on a surface. } 
{Contemp. Math.}  {\bf{96}} (1989), 103-108.



\bibitem{cohen1}
F. Cohen and D. Handel, \emph{$k$-regular embeddings of the plane}.
 Proc. Amer. Math. Soc. {\bf{72}} (1978), 201-204.



\bibitem{bundle1978}
F.R. Cohen, M.E. Mahowald and R.J. Milgram, \emph{The stable decomposition for the double loop space of a sphere.} {Proc. Sympos. Pure Math.}  {\bf{32}} (1978), 225-228.




\bibitem{mfdln2}
L.  Dyballa  and   S.  W.  Zucker,  \emph{IAN: Iterated Adaptive Neighborhoods for manifold learning and dimensionality estimation. } Neural Comput. {\bf  35} (2023),  453-524.  




\bibitem{farb08}
M. Farber,  \emph{Invitation to Topological Robotics.} Zurich Lectures in Advanced Mathematics,   
Zurich,  European Mathematical Society,  2008. 

\bibitem{Farber}
M. Farber,  \emph{Configuration spaces and robot motion planning algorithms.  }
  Combinatorial and Toric Homotopy, Lecture Notes Series, Institute for Mathematical Sciences, National University of Singapore  (2017),     263-303.  

\bibitem{simmfd}
M. Felisatti  and 
F.  Neumann,  \emph{Secondary theories for simplicial manifolds  and classifying spaces.  }
Geom.  Topol.  Monographs {\bf  11} (2007), 33-58. 


 \bibitem{annals}
 W.  Fulton and R.  MacPherson,  \emph{A  compactification of configuration spaces. }
 Ann.  Math.  {\bf  139}  (1994),  183-225.  
 


\bibitem{wulaoshi}
J. Grbi\'c, J. Wu, K.~L.  Xia  and G.-W. Wei,  \emph{Aspects of topological approaches for data science}.  Found.  Data Sci. {\bf  4}  (2022),   165-216. 


\bibitem{lin2}
A.  Grigor'yan,  Y.  Lin, Y.  Muranov  and  S.-T. Yau,  \emph{Homologies of path complexes and digraphs.}  arXiv (2013),   1207.2834.


\bibitem{lin6}
A.  Grigor'yan,   Y.  Lin, Y.  Muranov  and   S.-T.  Yau,  \emph{Path complexes and their homologies. } \emph{J.  Math.  Sci.}  {\bf 248}  (2020),   564-599.





\bibitem{handel1}
D. Handel, \emph{Obstructions to $3$-regular embeddings}.  Hoston J. Math. {\bf{5}} (1979), 339-343. 

\bibitem{handel2}
D. Handel, \emph{Some existence and nonexistence theorems for $k$-regular maps}.
 Fund. Math. {\bf{109}} (1980), 229-233. 



\bibitem{handel3}
D. Handel and J. Segal, \emph{On $k$-regular embeddings of spaces in Euclidean space}.
 Fund. Math. {\bf{106}} (1980), 231-237. 
 
 \bibitem{handel4}
D. Handel, \emph{Approximation theory in the space of sections of a vector bundle}. Trans. Amer. Math. Soc. {\bf{256}} (1979), 383-394.  



\bibitem{mfd-ln}
M.  Meila  and  H.  Zhang,  \emph{Manifold learning: what, how, and why}.  
Annu. Rev. Stat. Appl. {\bf 11}  (2024),  393-417.   
 

\bibitem{mapping2}
D. McDuff, \emph{Configuration spaces of positive and negative particles.}  {Topology} {\bf{14}} (1975), 91-107.

 




 



 \bibitem{parks}
A. D. Parks and S. L. Lipscomb, \emph{Homology and hypergraph acyclicity: a combinatorial in-
variant for hypergraphs.}  Naval Surface Warfare Center, 1991.
 

\bibitem{osaka}
 S. Ren,  \emph{Order of the canonical vector bundle over configuration spaces of projective spaces.} Osaka  J. Math. {\bf  54} (2017), 623-634.
 


\bibitem{forum}
S.  Ren,  \emph{Order of the canonical vector bundle over configuration spaces of spheres.}  {Forum  Math.}  {\bf 30} (2018), 1265-1277.  

 
 
 


\bibitem{jktr2}
S.   Ren,  C. Wu  and  J.   Wu,   \emph{Maps on random hypergraphs and random simplicial complexes.}  {J. Knot Theory Ramif. }  {\bf 31} (2022),      2250015.



\bibitem{jktr2023}
 S.  Ren, C.  Wu  and J. Wu,  \emph{Random hypergraphs, random simplicial complexes and their K\"unneth-type formulae.}  {J.  Knot  Theory Ramif.}
{\bf   32} (2023),   2350075.  



\bibitem{salvatore}
P. Salvatore, \emph{Configuration spaces on the sphere and higher loop spaces. } {Math. Z. }{\bf{248}} (2004), 527-540.  




\bibitem{singer}
I. Singer, \emph{Best approximation in normed linear spaces by elements of linear subspaces}, Springer-Verlag, Berlin and New York, 1970.



\bibitem{mfdln1}
J.  B.  Tenenbaum, V. D. Silva  and J. C. Langford,  \emph{ A global geometric framework for nonlinear dimensionality reduction.}  Science  {\bf 290} (2000),  2319-2323.


\bibitem{swy}
S.W. Yang, \emph{Order of the canonical vector bundle on $C_n(k)/\Sigma_k$.}  {Illinois J. Math. }{\bf{25}} (1981), 136-146.  

\bibitem{tams1}
A. Bianchi, J. Miller and J. C. H. Wilson, \emph{ Mapping class group actions on configuration spaces and the Johnson filtration. } Trans.  Amer.  Math.  Soc. {\bf  375} (2022), 5461-5489.

\bibitem{cph1}
A. Zomorodian and G. Carlsson, \emph{Computing persistent homology}. Discrete Comput. Geom. {\bf 33}(2) (2005), 249--274.

  \end{thebibliography}
  \end{document}